\documentclass[12pt]{amsart}
\usepackage{amsmath,amssymb,amsthm,mathrsfs,enumerate,comment,color,hyperref}
\usepackage{tikz-cd}
\newcommand{\SL}{\mathrm{SL}}
\newcommand{\GL}{\mathrm{GL}}
\newcommand{\Sp}{\mathrm{Sp}}
\newcommand{\algG}{\mathbf{G}}
\newcommand{\algLieG}{\mathrm{Lie}(\algG)}
\newcommand{\LieG}{\mathfrak{g}}
\newcommand{\Ad}{\mathrm{Ad}}
\newcommand{\LieT}{\mathfrak{t}}
\newcommand{\so}{\mathfrak{so}}
\newcommand{\Liesl}{\mathfrak{sl}}
\newcommand{\gl}{\mathfrak{gl}}
\newcommand{\algOq}{\mathbf{O}(q)}
\newcommand{\algO}{\mathbf{O}_n}
\newcommand{\algSO}{\mathbf{SO}_n}
\newcommand{\Orth}{\mathrm{O}}
\newcommand{\SO}{\mathrm{SO}}
\newcommand{\sltwok}{\mathfrak{sl}_2(\ratk)}
\newcommand{\sltwoR}{\mathfrak{sl}_2(\R)}

\newcommand{\gz}{\mathfrak{s}}

\newcommand{\N}{\mathbb{N}}
\newcommand{\Z}{\mathbb{Z}}
\newcommand{\resk}{\mathfrak{f}}
\newcommand{\field}{F}
\newcommand{\p}{\varpi}
\newcommand{\val}{\mathrm{val}}
\newcommand{\chr}{\mathrm{char}}

\newcommand{\diag}{\textrm{diag}}
\newcommand{\Hom}{\mathrm{Hom}}
\newcommand{\ratk}{k}
\newcommand{\ratks}{(k^\times)^2}
\newcommand{\R}{\mathcal{R}}
\newcommand{\foo}{\R}
\newcommand{\PP}{\mathcal{P}}
\newcommand{\ang}[1]{\ensuremath{\langle #1 \rangle}}
\newcommand{\noar}{\ensuremath{-1 \notin \ratk^2}} 
\newcommand{\scno}{\ensuremath{-1 \in \ratk^2}} 

\newcommand{\Gammamax}{{\Gamma_{\textrm{max}}}}
\newcommand{\LieTrip}{{\mathrm{d}\phi}}
\newcommand{\phiR}{{\phi_\R}}
\newcommand{\ep}{\varepsilon}
\newcommand{\Ip}{\mathcal{I}_\partt} 
\newcommand{\mat}[1]{\left[ \begin{matrix} #1 \end{matrix} \right]}
\newcommand{\X}{\mathsf{X}}
\newcommand{\Ximj}{\X_{i,j}}  
\newcommand{\Xipj}{\X_{i,-j}} 
\newcommand{\Xil}{\X_i^\ell} 
\newcommand{\Hi}{\mathsf{H}_i}
\newcommand{\Hj}{\mathsf{H}_j}
\newcommand{\HH}{\mathsf{H}}
\newcommand{\lconj}[2]{{}^{#1}#2}      
\newcommand{\kay}{\kappa}
\newcommand{\spn}{\mathrm{span}}
\newcommand{\la}{\langle}
\newcommand{\ra}{\rangle}

\newcommand{\labelstyle}[1]{\textsf{\textbf{#1}}}
\newcommand{\labeleven}{{\labelstyle{even}}}
\newcommand{\labelhyp}{{\labelstyle{hyp}}}
\newcommand{\labelpairs}{{\labelstyle{pairs}}}
\newcommand{\labelfours}{{\labelstyle{quad}}}
\newcommand{\labelani}{{\labelstyle{ani}}}
\newcommand{\labelsign}{{\labelstyle{sign}}}
\newcommand{\labeltriple}{{\labelstyle{trip}}}
\newcommand{\lab}{\labelstyle{L}}

\newcommand{\buil}{\mathcal{B}}
\newcommand{\cB}{\buil}
\newcommand{\apart}{\mathcal{A}}
\newcommand{\facet}{\mathcal{F}}

\newcommand{\set}[1]{\ensuremath{\{#1\}}}
\newcommand{\cL}{\mathcal{L}}

\newcommand{\cU}{\mathcal{U}}

\newcommand{\cV}{\mathcal{V}}

\newcommand{\iicM}{{\mathcal{M}^i}}
\newcommand{\iicMx}{\mathcal{M}^i_x}
\newcommand{\iicMxt}{\mathcal{M}^i_{x,t}}
\newcommand{\iicMxmtp}{\mathcal{M}^i_{x,-t+}}

\newcommand{\bU}{\mathbb{U}}

\newcommand{\bL}{\mathbb{L}}

\newcommand{\sR}{\mathscr{R}}
\newcommand{\sE}{\mathscr{E}}
\newcommand{\rU}{\mathrm{U}}

\newcommand{\LattD}{\mathrm{Latt}_{\R_D}}

\newcommand{\Num}[2]{\mathsf{N}(#2)_{#1}}  
\newcommand{\Numb}[2]{\mathsf{M}(#2)_{#1}}  
\newcommand{\Numc}[4]{\mathsf{T}(#4)_{#1,#2,#3}}  
\newcommand{\qtup}{\mathsf{q}}  
\newcommand{\cls}[1]{\overline{#1}}  
\newcommand{\da}[1]{\dim\cls{#1}}   
\newcommand{\W}{\mathcal{W}}  
\newcommand{\ani}{{n_\circ}} 
\newcommand{\anip}{{n'_\circ}} 
\newcommand{\aw}{\mathsf{u}}
\newcommand{\bw}{{\mathsf{u}'}}
\newcommand{\oap}{\chi}
\newcommand{\obp}{\chi}
\newcommand{\ai}{\chi_i}
\newcommand{\ao}{{\chi_1}}
\newcommand{\bo}{{\chi_2}}
\newcommand{\co}{{\chi_{3+}}}
\newcommand{\Hyp}{\mathcal{H}} 
\newcommand{\Quat}{\mathcal{Q}_4} 
\newcommand{\rp}{\rho\varpi}
\newcommand{\qani}{q_{\circ}}

\newcommand{\slmod}{U}
\newcommand{\mults}{M}
\newcommand{\slf}{\beta}
\newcommand{\wtf}{\gamma}
\newcommand{\partt}{\lambda}
\newcommand{\Partt}{\Lambda^e}
\newcommand{\Partte}{\Lambda^{ve}}
\newcommand{\Ppt}{P_{\partt,q}}
\newcommand{\Nilp}{\mathcal{N}}
\newcommand{\Nilps}{\mathcal{N}^{hyp}}
\newcommand{\Qtup}{\mathcal{Q}}

\theoremstyle{plain}
\newtheorem{theorem}{Theorem}[section]
\newtheorem*{maintheorem}{Theorem}
\newtheorem{lemma}[theorem]{Lemma}
\newtheorem{proposition}[theorem]{Proposition}
\newtheorem{corollary}[theorem]{Corollary}
\newtheorem{algorithm}[theorem]{Algorithm}

\theoremstyle{definition}

\theoremstyle{remark}
\newtheorem{example}{Example}

\numberwithin{equation}{section}

\begin{document}

\title[Nilpotent orbits of orthogonal groups]{Nilpotent orbits of orthogonal groups over $p$-adic fields, and the DeBacker parametrization}
\author{Tobias Bernstein} \address{Department of Mathematical Sciences, University of Alberta, Edmonton, Canada}
\email{tobias.toby.bernstein@gmail.com}
\author{Jia-Jun Ma}\address{School of Mathematical Sciences,
Shanghai Jiao Tong University,
800 Dongchuan Rd Shanghai, 200240, China}\email{hoxide@sjtu.edu.cn}
\author[Monica Nevins]{Monica Nevins${}^{1,2}$}\address{Department of Mathematics and Statistics, University of Ottawa, Ottawa, Canada}\email{mnevins@uottawa.ca}
\author{Jit Wu Yap}\address{Department of Mathematics, National University of Singapore, Block S17, 10 Lower Kent Ridge Road, Singapore 119076}\email{yapjitwu@gmail.com}
\thanks{${}^1$ This author's research is supported by a Discovery Grant from NSERC Canada.}\thanks{${}^2$ Corresponding author.}

\date{\today}

\begin{abstract}For local non-archimedean fields $\ratk$ of sufficiently large residual characteristic, we explicitly parametrize and count the rational nilpotent adjoint orbits in each algebraic orbit of orthogonal and special orthogonal groups.  We separately give an explicit algorithmic construction for representatives of each orbit. We then, in the general setting of groups $\mathrm{GL}_n(D)$, $\mathrm{SL}_n(D)$ (where $D$ is a central division algebra over $\ratk$) or classical groups, give a new characterisation of the ``building set'' (defined by DeBacker) of an $\mathfrak{sl}_2(\ratk)$-triple in terms of the building of its centralizer. Using this, we prove our construction realizes DeBacker's parametrization of rational nilpotent orbits via elements of the Bruhat-Tits building. \end{abstract}
\maketitle

\section{Introduction}

Rational, or arithmetic, nilpotent adjoint orbits of algebraic groups over a local field $\ratk$ arise in representation theory in several contexts.  For example, the Harish-Chandra--Howe character formula locally expresses a character of a representation as a linear combination of (Fourier transforms of) nilpotent orbital integrals.  As another example, the orbit method would parametrize representations by \emph{admissible} coadjoint orbits, with the admissible nilpotent orbits corresponding to core singular cases.

Algebraic, or geometric, nilpotent adjoint orbits can be thought of as those under the algebraic group over the algebraic closure of the local field. 
These orbits can be parameterized in multiple ways, including the Bala-Carter classification (extended to low characteristic by McNinch and others), weighted Dynkin diagrams, and partition-type classifications (for classical groups).  

The rational points of an algebraic orbit form zero or more rational orbits, and these can in principle be counted using Galois cohomology; yet it remained an open combinatorial problem to count these orbits for orthogonal groups.  Solving this is the first goal of this paper, in Section~\ref{S:count}.

Our second goal is to present an algorithm for generating  representatives for all rational nilpotent orbits of orthogonal and special orthogonal groups over $\ratk$, in the spirit of the one presented by Collingwood and McGovern in \cite{CMcG1993} over $\mathbb{R}$; our solution is presented in Section~\ref{S:representatives}. 

Our third and most important goal is to offer insight into the geometric para\-met\-rization of rational nilpotent orbits by elements of the Bruhat-Tits building of $G$ that was proposed by DeBacker in \cite{DeBacker2002}.  To this end we prove, in the more general setting of $G=\GL_n(D)$,  $\SL_n(D)$ (for $D$ a central division algebra over $\ratk$) and classical groups, that DeBacker's ``building set'' attached to a Lie triple can be identified with the building of the centralizer of that Lie triple.  This kind of ``functoriality result'' gives a coherent interpretation of the geometry of the DeBacker parametrization, and is presented in Section~\ref{S:functoriality}.

Finally, combining these results, we attach a representative of each nilpotent orbit to a facet of the building in our standard apartment, and prove that this gives an explicit realization of the DeBacker correspondence, in Section~\ref{S:minimal}.

\vspace{12pt}

Let us now summarize the results of this paper in more detail.

For symplectic, orthogonal or special orthogonal groups the algebraic nilpotent orbits can be parametrized by partitions.   
Let $\partt$ be a partition of $n$ in which even parts occur with even multiplicity and let $\mathcal{O}_\partt$ denote the corresponding nilpotent adjoint orbit of the algebraic group $\algO$.  
For $i\in \mathbb{N}$ let $\ai$ be the number of \emph{odd} parts of $\lambda$ whose multiplicity is exactly $i$.  Set $\co := \sum_{i\geq 3} \ai$. 

Let $(q,V)$ be a quadratic space; then $\algOq$ is a fixed $\ratk$-form of $\algO$ with respect to which we take the $\ratk$-points of this orbit.  Assume the characteristic of $\ratk$ is either zero or sufficiently large (see Section~\ref{S:nilpotent}) and that $p\neq 2$.  Then it is known that the rational orbits occurring $\mathcal{O}_{\partt}(\ratk)$ are parametrized by certain tuples $\qtup$ of quadratic forms (Theorem~\ref{T:mainthm}).  Our first result is to compute their number.

\begin{maintheorem}[Theorem~\ref{T:count}]
Let $(q,V)$ be a nondegenerate $n$-dimensional quadratic space of anisotropic dimension $\ani \leq n$. Then the number of $\ratk$-rational orbits under $\Orth(q)$ in $\mathcal{O}_{\partt}(\ratk)$ is
$$
\Numc{\ao}{\bo}{\co}{q} = \begin{cases}
\frac18 4^{\ao} 7^{\bo} 8^{\co} & \text{if $\co\geq 1$};\\
\frac18 4^{\ao} 7^{\bo}+ (2-\ani)2^{\ao-2} & \text{if $\co=0$, $\ao\geq 1$};\\
\lfloor \frac18 7^{\bo}\rceil + \ep_{\ani,\bo} & \text{if $\ao=\co=0$},
 \end{cases}
$$
where $\ep_{\ani,\bo}=0$ unless either $\ani=0$ and $\bo$ is even, or $\ani=4$ and $\bo$ is odd, in which cases $\ep_{\ani,\bo}=(-1)^{\bo}$.  The number of $\ratk$-rational orbits in $\mathcal{O}_{\partt}(\ratk)$  under $\SO(q)$ is the same unless $\ao=\bo=\co=\ani=0$, when there are two.
\end{maintheorem}

The arithmetic complexity of this answer stands in sharp contrast to the cases of $\SL(n)$ and $\Sp(2n)$.  The number of $\SL(n,\ratk)$-orbits in one algebraic nilpotent orbit of $\SL(n)$ is $\vert \ratk^\times/(\ratk^\times)^g\vert$ where $g$ is the gcd of the parts of the corresponding partition.  
The number of $\Sp(2n,\ratk)$ orbits in the algebraic orbit of $\Sp(2n)$ corresponding to a partition $\lambda$  is simply $4^{\ao}7^{\bo}8^{\co}$ where now $\ai$ counts the  number of \emph{even} parts of $\lambda$ with multiplicity exactly $i$ \cite{Nevins2011a}.

Next, for each pair $(\partt,\qtup)$ parametrizing a rational nilpotent orbit of $\SO(q)$, we choose, following Proposition~\ref{P:partition}, a partition $\Gamma$ of the set $\Ip :=\{(i,j) \mid i \in \partt, \, 1\leq j\leq m_i\}$, where $m_i$ denotes the multiplicity of $i$ in $\partt$.   We construct an explicit orbit representative and associated Jacobson-Morozov triple $\LieTrip$ by constructing subspaces of $V$ for each part of $\Gamma$ in Sections~\ref{S:hyp} to \ref{S:ve}.  This kind of explicit parametrization has many uses in representation theory, including: computing Fourier coefficients of automorphic forms as in \cite{Ahlenetal2018} and \cite{JiangLiu2015}; geometrizing invariant distributions coming from nilpotent orbits \cite{Christie}; and proving the motivic nature of Shalika germs as in \cite{FrechetteGordonRobson2015}, building on work of \cite{Diwadkar2008}.  Note that although determining a complete set of representatives in the case of special linear and symplectic groups is a direct generalization of the real case (see \cite{Nevins2011a}), orthogonal and special orthogonal groups present a special challenge, and this result is not straightforward.

Using a ``generalized Bala-Carter'' philosophy, DeBacker parametrized the rational nilpotent orbits of groups over, among others, local non-archimedean fields (with restrictions on residual characteristic) using the Bruhat-Tits building of the corresponding group.  The key construction is of a ``building set'' of a Lie triple $\{Y,H,X\}$, denoted $\buil(Y,H,X)$.  Namely, the DeBacker parametrization attaches to each rational nilpotent orbit $\mathcal{O}$ one or more  \emph{degenerate pairs}, which for our purposes we may take to be pairs $(\facet, X)$, where 
$\facet \subset \buil(Y,H,X)$ where $\facet$ is a facet and $\{Y,H,X\}$ is  a Lie triple extending a representative $X$ of $\mathcal{O}$ (see Section~\ref{S:minimal}).  When $\facet$ is maximal in $\buil(Y,H,X)$, the pair is called \emph{distinguished}; \emph{associativity classes} of distinguished pairs are in bijection with rational nilpotent orbits \cite{DeBacker2002}.  The challenge inherent in this description is that it does not suffice to work within a single apartment: a facet $\facet$ of the building may be maximal in $\buil(Y,H,X)\cap \apart$ without being distinguished! 

We prove the following general result in Section~\ref{S:functoriality}.  Here $D$ denotes a central division algebra over $\ratk$, $\R$ denotes the integer ring of $\ratk$, and $p$ is sufficiently large (see Section~\ref{S:nilpotent}).

\begin{maintheorem}[Theorem~\ref{T:building}]
  Suppose $G$ is $\GL_n(D)$, $\SL_n(D)$ or a classical group, and suppose
  $\LieTrip=\set{Y,H,X}$ is a Lie triple in $\LieG$, with corresponding homomorphism $\phi \colon \SL_2(\ratk) \to G$.   Let $G^\phi$ be the centralizer of $\phi(\SL_2(\ratk))$ in $G$.  
  Then $\buil(Y,H,X)=\cB(G)^{\phi(\SL_2(\R))}$ and there is a $G^\phi$-equivariant identification 
  \[
  \cB(G)^{\phi(\SL_2(\R))} = \cB(G^\phi). 
  \]
\end{maintheorem}

An immediate consequence is a formula for the dimension of all the maximal facets in $\buil(Y,H,X)$, whence it suffices to produce a pair $(\facet,X)$ attached to $\mathcal{O}$ of the correct dimension in order to deduce that it is distinguished.  We apply this approach to prove the correctness of our parametrization for orthogonal and special orthogonal groups in Proposition~\ref{P:dimension} and Theorem~\ref{T:DeBackercorrespondence}.

DeBacker's parametrization has only been explored in a handful of cases, including  \cite{Nevins2011a} for the special linear and symplectic groups, where the dimensions of the maximal facets were established via  combinatorial arguments.  

The current paper initially arose from an NSERC USRA project of T.~Bernstein, working under the supervision of M.~Nevins, on counting the number of rational nilpotent orbits for orthogonal groups.  M.~Nevins complemented this with a parametrization of these orbits and circulated a preprint, whereupon the J.-J.~Ma and J.W.~Yap shared their report \cite{Yap2018}.  In it, J.W.~Yap, working under the supervision of J.-J.~Ma, constructs distinguished representatives of rational nilpotent orbits of split even orthogonal groups (correcting also an error in the proof of \cite[Theorem~4]{Nevins2011a}), and proves Theorem~\ref{T:building} in that case.  They had also gone on to prove Theorem~\ref{T:building} as it appears here.  J.-J.~Ma would like to thank J.K.~Yu for many helpful discussions on this topic.

Several interesting questions remain open.

For one, although the proof of Theorem~\ref{T:building} is currently restricted to particular groups, J.-J.~Ma has shown separately the existence of the map $\sR\colon \buil(G)^{\phi(\SL_2(\R))} \to \buil(G^\phi)$ \eqref{eq:R} for any connected semisimple $\algG$ of adjoint type over $\ratk$, and conjectures that there should be a natural inverse map $\sE$.

For another, it is an open question to determine known invariants of rational nilpotent orbits in terms of the data of their DeBacker parametrization.  Together with \cite{Nevins2011a}, we now have the complete parametrization for all split classical groups, which opens the possibilities for study.  Part of the problem would be to give a combinatorial description of the associativity classes of facets in $\buil$, and more particularly of the $r$-associativity classes for  each $r\in \mathbb{R}$, which are greater in number and offer a finer parametrization.  

Our counting results rely on Jacobson-Morozov theory to describe the nilpotent orbits, and thus entail a restriction on the characteristic of $\ratk$.  It would be interesting to count rational orbits, and give explicit representatives, in these missing cases.

\vspace{12pt}

This paper is organized as follows.   In Section~\ref{S:notation} we establish our notation and some necessary results about quadratic forms.  In Section~\ref{S:nilpotent} we give the constraints on the characteristic and residual characteristic for the union of the results in this paper, and review key facts about nilpotent orbits and the orthogonal groups.  Section~\ref{S:count} is devoted to the proof of Theorem~\ref{T:count}, counting the number of rational orbits.  In Section~\ref{S:representatives}, we present an algorithm for generating representatives of each orbit.  To do so explicitly, we set the notation for root vectors in Section~\ref{SS:notation} and describe the overall strategy in Section~\ref{SS:strategy}, with details for each of the subcases in Sections~\ref{S:hyp} to \ref{S:ve}.  In Section~\ref{S:functoriality} we revert to the case of general $G$ and briefly recall the DeBacker parametrization, before proving Theorem~\ref{T:building}.
In Section~\ref{S:minimal} we attach to each of our orbit representatives a distinguished pair, thus establishing a new dictionary from the partition-based to the building-based parametrizations of rational nilpotent orbits for orthogonal and special orthogonal groups.

\section{Notation and the Witt group}\label{S:notation}

Let $\ratk$ be a local non-archimedean field of residual characteristic $p\neq 2$, with integer ring $\R$ and maximal ideal $\PP$ generated by a uniformizer $\varpi$.  Denote by $\resk$ the residue field of $\ratk$.  Let $\rho$ be a fixed nonsquare in $\R^\times$ with image $\rho_{\resk}$ in $\resk^\times$.   

The following theory is concisely presented in \cite{TBernstein2015} and based on \cite[Chapter 1]{Lam2005}. A quadratic space $(q,V)$ over a field $\field$ such that $\textrm{char}(F)\neq 2$ is a finite-dimensional vector space $V$ over $\field$ equipped with a regular quadratic form $q$; when needed, its associated (nondegenerate) bilinear form is denoted $B_q$, a matrix form is $M_q$, and the dimension of $V$ is $\deg(q)$, the degree of $q$.  Denote by $\Hyp$ the quadratic hyperbolic plane.  

If $(q,V)$ and $(q',V')$ are two quadratic spaces we write $q \cong q'$ if they are isometric and $q\simeq q'$ if the isometry classes of the quadratic forms $q$ and $q'$  differ by a sum of hyperbolic planes.  Then $\simeq$ defines an equivalence relation on the monoid of nondegenerate quadratic forms, and the resulting quotient is the \emph{Witt group} $\W_\field$ of $\field$ with trivial element denoted $\overline{0}$ or $\Hyp$.  
Write $\cls{q}$ for the image of $q$ in $\W_\field$, which we can identify up to isometry with the anisotropic kernel $\qani$ of $q$.  Then $\da{q}:=\deg(\qani) = \ani$ is the \emph{anisotropic dimension} of $q$.  

Each quadratic space $(q,V)$ admits a basis relative to which $q$ is diagonalized; in this case we write $q = \ang{a_1,\cdots, a_n}$ for some $a_i \in \field^\times$ but even up to permuting and scaling each coordinate by elements of $(\field^\times)^2$ this representation of $q$ is not necessarily unique.  

If $\field = \resk$, then since $p>2$ we have $\W_\resk = \{\Hyp, \ang{1}, \ang{\rho_\resk}, \ang{1,-\rho_{\resk}}\}$, which has the structure of $\W_\resk^-\cong (\Z/2\Z)^2$ if $\scno$ (that is, if $p$ is congruent to $1$ mod $4$) and of $\W_{\resk}^+\cong \Z/4\Z$ otherwise.  The identification of sets $\iota \colon \W_{\resk}^-\to \W_{\resk}^+$ is thus not a homomorphism but it is easy to check that it satisfies the very useful property that for all $\aw,\bw\in \W_{\resk}^-$,
\begin{equation}\label{E:iota1}
\dim(\iota(\aw)-\iota(\bw))=\dim(\iota(\aw-\bw)).
\end{equation}
If $\field = \ratk$, then the map $\rho_\resk \mapsto \rho$ induces a well-defined injection $i\colon \W_{\resk}\to \W_\ratk$.  In fact, the map which sends $(\aw,\bw)\in \W_{\resk}^2$ to the class of $i(\aw) \oplus \varpi i(\bw)$ defines an isomorphism $\W_\ratk \cong \W_{\resk}^2$.  (We may write $\W_\ratk^\pm$ when we want to specify the group structure.)

 We list the distinct elements of $\W_\ratk$ in the second and third columns of Table~\ref{Table:Witt}, in terms of the favoured representatives $\{1,\rho, \varpi, \rho\varpi\}$ for $\ratk^\times/(\ratk^\times)^2$, and grouped by their anisotropic dimension (given in the first column).   Write $\Quat$ for the unique class of anisotropic dimension $4$, which is the quaternionic class.   We now collect some facts needed for Section~\ref{S:count}.
 
\begin{table}
	\begin{center}
		\begin{tabular}{|c|c c|c|c|}
		\hline
		$\da{q}$	 & \multicolumn{2}{c|}{Representative for $\cls{q}$} & Number of & Common\\
			&\scno			&\noar		&Choices &Representative	\\	\hline \hline
		$0$			&\ang{a,a} 		&$\ang{1,\rho}=\ang{\p,\rp}$		&$4$&$\Hyp = \ang{a,-a}$	\\ \hline	
		$1$			&\ang{1}		&\ang{1}	&$1$	&\ang{1}	\\	
					&\ang{\rho}		&\ang{\rho}	&$1$	&\ang{\rho}	\\	
				&	\ang{\p}		&\ang{\p}	&$1$	&\ang{\p}	\\	
				&	\ang{\rp}		&\ang{\rp}	&$1$	&\ang{\rp}	\\	\hline
		$2$		&	\ang{1,\rho}		&\ang{1,1}=\ang{\rho,\rho}	&$2$	&\ang{1, -\rho} \\	
				&	\ang{1,\p}		&\ang{1,\p}	&$2$	&\ang{1,\p}	\\	
				&	\ang{1,\rp}		&\ang{1,\rp}	&$2$	&\ang{1,\rp}	\\	
				&	\ang{\rho, \p}	&\ang{\rho,\p}	&$2$ &\ang{\rho,\p}	\\	
				&	\ang{\rho, \rp}	&\ang{\rho, \rp} &$2$	&\ang{\rho,\rp}	\\	
				&	\ang{\p,\rp}		&\ang{\p,\p}=\ang{\rp,\rp} &$2$	&\ang{\p, -\rp}	\\ \hline
		$3$		&	\ang{1, \rho, \p}	&\ang{1,1, \p}	&$6$&\ang{1, -\rho, \p}	\\
				&	\ang{1, \rho, \rp}	&\ang{1, 1, \rp}	&$6$&\ang{1, -\rho, \rp}	\\ 
				&	\ang{1, \p, \rp}	&\ang{1, \p, \p} &$6$	&\ang{1, \p, -\rp}	\\
				&	\ang{\rho, \p, \rp}	&\ang{\rho, \p,\p} &$6$	&\ang{\rho, \p, -\rp}	\\  \hline
		$4$		&	\ang{1,\rho,\p,\rp} &\ang{1, 1, \p, \p} &$24$ &$\Quat = \ang{1,-1,\p,-\p}$ \\ \hline
		\end{tabular}
	\caption{Representatives of elements of $\W_\ratk$ (in two forms: simple ones dependent on the sign of $-1$ in $\ratk$, and more complex ones which are independent thereof), together with the number of choices of distinct diagonal representatives of each up to $\ratks$.}
	\label{Table:Witt}
	\end{center}
\end{table}

  \begin{lemma} \label{L:invariant}
  Let $\ratk$ be a local non-archimedean field of odd residual characteristic.  
  \begin{enumerate}[1.]
  \item The number of isometry classes of quadratic forms of degree $n$ is $4$ if $n=1$, $7$ if $n=2$, and $8$ if $n\geq 3$.
\item   The number of choices of distinct diagonal representations of each anisotropic form or hyperbolic plane, counting order but modulo $\ratks$, is an invariant of the anisotropic dimension and is independent of the class of $p$ mod $4$.
\item  The map $\iota$ extends to a bijection $\iota \colon \W_{\ratk}^- \to \W_{\ratk}^+$ such that for all $\aw,\bw \in \W_{\ratk}^-$, \begin{equation}\label{E:difference}
\dim(\iota(\aw)-\iota(\bw))=\dim(\iota(\aw-\bw)).
\end{equation}
\end{enumerate}
  \end{lemma}
  
\begin{proof}
The first statement is well-known, but can also be inferred from Table~\ref{Table:Witt} directly.
We have recorded the number of choices of distinct diagonal representatives for each class of anisotropic form or hyperbolic plane, counting order but modulo scaling in each factor by $\ratks$, in the fourth column of Table~\ref{Table:Witt}; this establishes the second assertion.  The map $\iota$ extends via the isomorphisms $i \colon \W_\ratk^\pm  \to (\W_\resk^\pm)^2$.  Since $\dim(i(\aw)\oplus \varpi i(\bw))= \dim(\aw)\oplus \dim(\bw)$ for any $\aw,\bw \in \W_{\resk}$, \eqref{E:difference} follows from \eqref{E:iota1}.  For convenience, we have recorded the common representatives defining the map $\iota$ in the last column of Table~\ref{Table:Witt}.
\end{proof}

We say that two tuples of quadratic forms $(q_1,q_2, \ldots, q_s)$ and $(q_1', q_2', \ldots, q_{s'}')$ are  \emph{isometric} if $s=s'$ and for all $i$, $q_i\cong q_i'$.  Let $\qtup = [q_1, q_2,\ldots, q_s]$ denote the corresponding  isometry class; then $\cls{\qtup}=\cls{q_1}+\cdots+\cls{q_s}$ is a well-defined element of $\W_\ratk$.
We say that $\qtup$ represents $q$ if $\cls{\qtup}=\cls{q}$.

Now let $\partt$ be a partition of $n$.  The multiplicity $m_i$ of $i$ in $\partt$ is number of times $i$ occurs in $\partt$.  Let $\Partt(n)$ denote the set of partitions of $n$ in which even parts occur with even multiplicity.  For $\partt \in \Partt(n)$, let $\Qtup_\partt = \{\Hyp\}$ if $\partt$ has no odd parts; otherwise, let $j_1<j_2<\ldots <j_s$ denote its distinct odd parts and let
$$
\Qtup_{\partt} = \{ \qtup := [q_{1},\ldots, q_{s}] \mid \text{for each $i$, $q_{i}$ is a quadratic form of degree $m_{j_i}$} \}
$$
be the set of isometry classes of $s$-tuples of quadratic forms of the stated degrees. 

Given a quadratic space $(q,V)$ of degree $n$, we set
\begin{equation}\label{Nqn}
\Nilp(\cls{q},n) = \{ (\partt, \qtup) \mid \partt\in \Partt(n), \qtup \in \Qtup_\partt \text{\; such that \;} \cls{\qtup} = \cls{q}\}.
\end{equation}
If $n$ is even, let $\Partte(n)\subset \Partt(n)$ be the subset of partitions of $n$ which have no odd parts; these are called \emph{very even partitions}.  If $\cls{q}=\cls{0}$ then to each very even partition we attach two distinct copies of $\Hyp$, to give
\begin{equation}\label{Nn}
\Nilps(n) = \Nilp(\cls{0},n) \sqcup  \{(\partt,\Hyp') \mid \partt \in \Partte(n)\}.
\end{equation}

\section{Lie triples, nilpotent adjoint orbits, and the orthogonal group}\label{S:nilpotent}
Let $\algG$ be a semisimple algebraic group defined over $\ratk$ and $\algLieG$ its Lie algebra.   Let $h$ be the maximal value of the Coxeter number of any irreducible component of the root system of $\algG$.  We assume that the characteristic of $\ratk$ is either zero, or else is greater than $3(h-1)$.  This hypothesis implies, by \cite[\S 5.5]{Carter1985}, that each Lie triple defined below  lifts to a unique homomorphism of algebraic groups $\phi \colon {\mathbf{\SL_2}}\to \algG$, and that the degree of nilpotency of any nilpotent element is less than $p$.  Hence, if $\chr(\ratk)=p$, then for each nilpotent element $X \in \algLieG$, we have $X^{[p]}=0$ where this denotes the $p$-operation on the restricted Lie algebra $\algLieG$.  We note that by the work of G.~McNinch, including particularly \cite{McNinch2005}, these hypotheses on $p$ can often be weakened via the theory of optimal $\SL_2$-homomorphisms.

\subsection{Lie triples}
Let $G=\algG(\ratk)$ and $\LieG=\algLieG(\ratk)$.  A Lie triple is a nonzero set $\{Y,H,X\}\subset \algLieG$ such that $[H,X]=2X$, $[H,Y]=-2Y$ and $[X,Y]=H$. Then  Jacobson-Morozov theory \cite[VIII,\S11]{Bourbaki1975}, \cite{McNinch2005} 
asserts a bijection between the nonzero nilpotent orbits of $\algG$ on $\algLieG$ (respectively, of $G$ on $\LieG$) and conjugacy classes of Lie triples in $\algLieG$ under $\algG$ (respectively, in $\LieG$ under $G$), given by associating the triple to the orbit of its nilpositive element $X$.  Moreover, by \cite[\S 5.5]{Carter1985} there is a group homomorphism $\phi \colon {\mathbf{\SL_2}} \to \algG$ defined over $\ratk$ for which 
\begin{equation}\label{E:dphi}
\LieTrip\left( \mat{0&0\\1&0} \right)=Y,\ \LieTrip\left( \mat{1&0\\0&-1} \right)=H, \;\text{and}\; \LieTrip\left( \mat{0&1\\0&0} \right)=X.
\end{equation}
We often denote the Lie triple $\{Y,H,X\}$ by $\LieTrip$.

Suppose $V$ is a $G$-module; then the subgroup $\phi(\SL_2(\ratk))$ (or equivalently its Lie algebra $\LieTrip$) decomposes $V$  into pairwise orthogonal isotypic components.  Let $\slmod_i$ denote the unique irreducible $\sltwok$-module of degree $i$, and set $\mults^i := \Hom_{\sltwok}(\slmod_i,V)$.  Then for each $i$ the map sending $(u,T)\in \slmod_i\otimes \mults^i$ to $T(u)\in V$ induces an isomorphism
\begin{equation}\label{decompgen}
V = \bigoplus_{i\in \N} V_i \cong \bigoplus_{i \in \N}\slmod_i \otimes \mults^i,
\end{equation}
where $V_i$ is the $\slmod_i$-isotypic component of $V$.  Since elements of $G^\phi$ commute with $\LieTrip$, each space $\mults^i$ is a $G^\phi$-module.

In Section~\ref{S:functoriality}, we have need of the preceding in the more general setting of $\SL_2(\R)$-modules.  For each $n<p$ let
$\cU_n$ be the irreducible module of $\SL_2(\foo)$ of rank
$n+1$. Then $\slmod_n = \cU_n\otimes_\foo \ratk$ and $\bU_n := \cU_n \otimes_\foo \resk := \cU_n/\PP\cU_n$.  For any $\foo$-lattice $\cL$, let $L :=
\cL\otimes_\foo \ratk$ and $\bL := \cL\otimes_{\foo} \resk$. 

\begin{lemma}\label{lem:key}
  Suppose $\cL$ is an $\foo$-lattice with $\SL_2(\foo)$-action given by $\phi$.
  Then
  \[
    \cL \cong \bigoplus_{i< p}\cU_i \otimes_\foo \iicM
  \]
  where $\iicM = \Hom_{\SL_2(\foo)}(\cU_i, \cL)$ is an $\foo$-lattice in $\mults^i$.
\end{lemma}
\begin{proof}
  By the assumption on $p$, the $\SL_2(\resk)$-module $\bL$ 
  is semisimple, and therefore $\bL \cong \bigoplus_{i<p} \bU_i \otimes
  \Hom_{\SL_2(\resk)}(\bU_i, \bL)$.
 It now follows from \cite[Proposition 5.3.1]{McNinch2018} that this isomorphism lifts to the level of $\SL_2(\R)$-modules.
\end{proof}

\subsection{Centralizers} \label{SS:centralizers}
Under the hypotheses on $\ratk$, the centralizer $G^\phi$ of $\phi(\SL_2(\ratk))$ in $G$ coincides with $G^{\LieTrip}$, the stabilizer under the adjoint action of the Lie triple $\LieTrip$.   Let us describe the structure of $G^\phi$ for a large class of groups.

Let $D$ be a central division algebra over $\ratk$ and suppose $V$ is additionally a right $D$-module.  Then $\mults^i$ inherits the structure of a right $D$-module.  In fact, if $G=\GL(V,D)$ then $G^\phi \cong \prod_{i}\GL(\mults^i,D)$.  

Now suppose further that $V$ is equipped with a non-degenerate $G$-invariant sesquilinear form $F$.  Let  $\slf_i$ denote the the unique (up to scaling) $\sltwok$-invariant nondegenerate bilinear form on $\slmod_i$.  It is symplectic if $i$ is even and symmetric if $i$ is odd.  In this latter case, fix a choice of form such that $\slf_i(\cU_i,\cU_i) = \R$ and $q_{\slf_{i}} \cong \ang{1} \oplus \Hyp^{\oplus k}$ where $i=2k+1$.  

Then under the isomorphism \eqref{decompgen}, for each $i$, $F$ and $\slf_i$ induce a nondegenerate $G^\phi$-invariant form $\wtf^i$ on $\mults^i$ which satisfies, for each $u, u' \in \slmod_i$ and each $T,T'\in \mults^i$, that $F(T(u), T'(u'))=\slf_i(u,u')\wtf^i(T,T')$.  
It follows that if $G$ is the group $\rU(V,F)$ of isometries of $(V,F)$, then $G^\phi \cong \prod_i \rU(\mults^i,\wtf^i)$.

In the special case that $D=\ratk$ and $F$ is a symmetric bilinear form, we have that $\wtf^i$ is symplectic for each even $i$ and symmetric for each odd $i$.  Writing $O(q):=\rU(V,B_q)$ for any quadratic space $(V,q)$ and $\Sp(n):=\rU(V,F)$ for any symplectic space $(V,F)$ of dimension $n$, we have
\begin{equation}\label{E:cent}
\Orth(q)^\phi \cong \prod_{\text{$i$ odd: $\mults^i \neq 0$}}\Orth(q_{\wtf^i}) \times \prod_{\text{$i$ even: $\mults^i \neq 0$}} \Sp(\dim(\mults^i)).
\end{equation}

\subsection{Nilpotent orbits of orthogonal groups and algebras}
For the remainder of this section and until Section~\ref{S:functoriality}, let $D=\ratk$ and $F=B_q$ be symmetric, so that $(q,V)$ is an $n$-dimensional quadratic space over $\ratk$.   
The \emph{special orthogonal Lie algebra} is 
$$
\so(q) = \{X \in \Liesl(n,\ratk) \mid \lconj{t}X M_q + M_qX=0\}.
$$
Observe that $\so(q) = \so(\alpha q)$ for any $\alpha \in \ratk^\times$, so from Table~\ref{Table:Witt} we infer there is a single isomorphism class of Lie algebra for each anisotropic dimension, except for $\dim \cls{q}=2$.  In this latter case, by the Kneser-Tits classification \cite{Tits1979}, there are two isomorphism classes, corresponding to Lie algebras splitting over a ramified or an unramified extension respectively.  The \emph{orthogonal group} is
$$
\Orth(q) = \{g \in \GL(n,\ratk) \mid \lconj{t}gM_qg=M_q\}
$$
and it contains  $\SO(q)$ as the index-two subgroup of elements of determinant equal to $1$.  These groups are compact if and only if $q$ is anisotropic.  We think of them as the $\ratk$-points of the corresponding inner forms of the algebraic groups $\algO$ and $\algSO$, respectively.  

Given a geometric nilpotent orbit $\mathcal{O}$ under the algebraic group $\algOq \cong \algO$, then its set of rational points $\mathcal{O}(\ratk)$ may be empty, or may decompose as a union of one or more rational nilpotent orbits.  In the latter case, using the arguments of \cite[Prop 4.1]{Nevins2002}, one can deduce that the set of rational orbits is in bijection with 
the kernel of the map of pointed sets in Galois cohomology
\begin{equation}\label{E:ker}
\alpha \colon H^1(\ratk, \algOq^\phi) \to H^1(\ratk, \algOq)
\end{equation}
where $\LieTrip$ is an $\Liesl_2$-triple for a base point of $\mathcal{O}$ and $\algOq^\phi$ is its centralizer.  

By \eqref{E:cent}, the algebraic group $\algOq^\phi$ is a product of symplectic and orthogonal groups.  For a group $U$ preserving a nondegenerate $m$-dimensional bilinear form, $H^1(\ratk,U)$ counts the number of $\ratk$-isometry classes of forms of degree $m$; thus it is trivial if the form is symplectic, and if the form is symmetric, it has order $4$, $7$ or $8$ if $m$ is $1$, $2$ or at least $3$, respectively.   The kernel of the map \eqref{E:ker} can thus parametrized by tuples of quadratic forms whose sum is equivalent to the chosen form $q$ in the Witt group.
This correspondence is made explicit in the proof of the following theorem,  which is a known result; for example, 
for $\mathbb{R}$ and $\mathbb{C}$ see \cite[Ch. 9]{CMcG1993} and for extensions of $\mathbb{Q}_p$ see \cite[I.6]{Waldspurger2001}.  

\begin{theorem}\label{T:mainthm}
Let $(q,V)$ be a nondegenerate $n$-dimensional quadratic space over $\ratk$.  The nilpotent $\Orth(q)$ orbits on $\so(q)$ are parametrized by the set $\Nilp(\cls{q},n)$.  If $\cls{q}\neq \cls{0}$, then the nilpotent $\SO(q)$ orbits coincide with those under $\Orth(q)$ but otherwise, $n$ is even and the nilpotent $\SO(q)$-orbits are parametrized by the set $\Nilps(n)$.
\end{theorem}

\begin{proof}
Let $X\in \so(q)\setminus\{0\}$ be nilpotent and let $\LieTrip = \{Y,H,X\}$ be a corresponding Lie triple.  Decompose $V$ into isotypic components under the corresponding action of $\sltwok$ (or equivalently under its lift $\phi(\SL_2(\ratk))$) as in \eqref{decompgen}; then the nonzero summands are indexed by parts $i$ of a partition $\partt$ of $n$ in which each part $i$ occurs with multiplicity $m_i = \dim(\mults^i)$.  Hence $\partt\in \Partt(n)$, since even parts correspond to symplectic forms.
If $X=\{0\}$ set $\LieTrip=\{0\}$ and $\mults^1=V$ in \eqref{decompgen}, so that $\partt=[1^n]$.  
When $i$ is an odd part of $\partt$, we have $q_{\slf_i}\simeq \langle 1\rangle$.  Thus writing $q_i := q|_{V_i}$ and noting that $q_{\slf_i}\otimes q_{\wtf^i} \cong q_i$, we conclude that $(q_{\wtf^i},\mults^i) \simeq (q_i,V_i)$ in $\W_\ratk$.  On the other hand, when $i$ is an even part of $\partt$, $\slf_i\otimes \wtf^i$ is a split quadratic space, whence $(q_i,V_i) \simeq \Hyp$ in $\W_\ratk$.

%

Consequently, the tuple $\qtup = [q_{\wtf^j} \mid \text{$j$ odd, $m_j \neq 0$}]$ lies in $\Qtup_\partt$.  Since $\oplus_{i\in \N}q_i \simeq q$, the pair $(\partt,\qtup)$ lies in $\Nilp(\cls{q},n)$.  Since $G^\phi$ acts by isometries on the spaces $\mults^i$, this map $\LieTrip \to \Nilp(\cls{q},n)$ lifts to a well-defined map on nilpotent orbits.  It is surjective, since each element of $\Nilp(\cls{q},n)$ defines a decomposition \eqref{decompgen} of $(q,V)$ (uniquely up to isometry) and consequently a subgroup of $\Orth(q)$ isomorphic to $\SL_2(\ratk)$.



To show this map is bijective, suppose $\LieTrip$ and $\LieTrip'$ are two Lie triples in $\so(q)$ and $V = \oplus_i V_i$ and $V = \oplus_i V_i'$ are the corresponding decompositions of $V$ into isotypic components.  Then these decompositions are isometric if and only if $\LieTrip$ and $\LieTrip'$ are conjugate via an element of $\Orth(q)$.  As $\Nilp(\cls{q},n)$ is in bijection with the set of isometry classes of such decompositions, the first statement of the theorem follows.

To understand the $\SO(q)$ orbits, suppose that $g\in \Orth(q) \setminus \SO(q)$ gives  $\Ad(g)\LieTrip = \LieTrip'$.  
From \eqref{E:cent}, and that the symplectic factors have determinant $1$, we conclude that $\Orth(q)^\phi$ contains an element $h$ of determinant $-1$ if and only if $\partt$ contains at least one odd part, in which case $gh \in \SO(q)$ and $\Ad(gh) \LieTrip = \LieTrip'$, showing that the $\Orth(q)$ and $\SO(q)$ orbits coincide.

If, however, $\dim(V_i)$ is even for all $i$, then no such $h$ exists.  In this case, each $V_i$ is a split quadratic space and so $(q,V)$ is a sum of hyperbolic planes, whence $\cls{q}=0$.  Since $\Orth(q)^\phi = \SO(q)^\phi$ in this case, and $\SO(q)$ has index two in $\Orth(q)$, we deduce that each of the $\Orth(q)$-orbits corresponding to $\partt\in \Partte(n)$ decompose as a disjoint union of two $\SO(q)$-orbits.
\end{proof} 

Note that for $i$ odd, where convenient, we can and do identify $\mults^i$ with the $0$-weight space of $V_i$.

\section{Counting rational nilpotent orbits}\label{S:count}
Let $\partt \in \Partt(n)$. 
For $i \in \mathbb{N}$ let $\ai$ denote the number of odd parts $j$ of $\partt$ occurring with multiplicity exactly $i$, and set $\co=\sum_{i=3}^n \ai$. 
Let $X$ be an element of the algebraic orbit $\mathcal{O}_\partt$, $\LieTrip$ an associated Lie triple, and $\phi$ the corresponding homomorphism $\phi \colon \SL_2\to \algOq$.
From the form of $\algOq^\phi$ in \eqref{E:cent}, and Lemma~\ref{L:invariant}, we deduce that
$$
\vert H^1(\ratk, \algOq^\phi) \vert = 4^{\ao}7^{\bo}8^{\co},
$$
whereas $\vert H^1(\ratk,\algOq)\vert \in \{4,7,8\}$, depending on $n$.   Thus from the discussion preceding the statement of Theorem~\ref{T:mainthm}, if $n\geq 3$ one expects for each choice of $\ratk$-form $\algOq$ about $\frac18(4^{\ao}7^{\bo}8^{\co})$ rational orbits in $\mathcal{O}_\partt(\ratk)$, with some variation depending on $\partt$ and $q$.

On the other hand, Theorem~\ref{T:mainthm} gives a direct means of counting the number of rational orbits in $\mathcal{O}_\partt(\ratk)$: they are parametrized by $\Ppt=\{ \qtup \in \Qtup_{\partt}\mid  \cls{\qtup} = \cls{q}\}$. 
That is to say, it suffices to count the number of isometry classes of tuples (of degrees prescribed by the multiplicities of the odd parts in $\partt$) that represent $\cls{q}$.  This is a nontrivial counting problem, and the subject of this section.

We begin with the simple case that each odd part of $\partt$ has multiplicity equal to $1$. 

\begin{lemma} \label{L:1dcase}
Let $\aw\in \W_\ratk$ and let $\ani$ be its anisotropic dimension $\ani = \dim(\aw)$.  Let $\ao\in\mathbb{N}_+$ have the same parity as $\ani$.
The number $\Num{\ao}{\aw}$ of isometry classes of $\ao$-tuples of degree-one quadratic forms representing $\aw$ is
\begin{equation} \label{n1}
\Num{\ao}{\aw} = \frac18 4^{\ao} + (2-\ani)2^{\ao-2}.
\end{equation}
\end{lemma}

\begin{proof}  For ease of notation set $\oap:=\ao$.  
We prove this formula by induction on even and odd $\oap$, respectively.   When $\oap=1$, we have $N(\aw)_1=1$ if $\ani=1$ but $N(\aw)_1=0$ if $\ani=3$, so  \eqref{n1} holds.  When $\oap=2$, there are $16$ distinct isometry classes of pairs of quadratic forms.  By Lemma~\ref{L:invariant}, regardless of the sign of $-1$ in $\ratk$, each of the six anisotropic quadratic forms $\aw$ with $\ani=2$ is represented by exactly two such pairs, accounting for $12=6\times 2$ pairs;  the remaining four pairs represent the hyperbolic plane (which has $\ani=0$). In particular no pair can represent $\Quat$ (which has $\ani=4$).  This count agrees with \eqref{n1} for $\oap=2$ and each $\ani\in \{0,2,4\}$.  Thus \eqref{n1} holds for $\oap\in \{1,2\}$ and all $\aw$ with $\ani$ of the same parity as $\oap$.

Suppose now that $\oap>2$ and that $\Num{\oap-2}{\bw}$ satisfies \eqref{n1} for all $\bw \in \W_\ratk$ such that $\anip:=\dim(\bw)$ has the same parity as $\oap$.  In particular, since the right side of \eqref{n1} depends only on the anisotropic dimension, we may define $\Num{\oap-2}{\anip}:=\Num{\oap-2}{\bw}$.  

Let $\aw \in \W_\ratk$ and suppose it is represented by a $\oap$-tuple of degree-one quadratic forms $\qtup_{\oap} = [q_1,q_2,\qtup_{\oap-2}]$ where $\qtup_{\oap-2}$ denotes an $(\oap-2)$-tuple.  Set $\bw = \cls{\qtup_{\oap-2}}$; then $\bw=\aw-\cls{\la q_1,q_2\ra} \in \W_\ratk$.    
Set $\ani=\dim(\aw)$ and $\anip=\dim(\bw)$; then necessarily $\anip \in \{\ani, \ani\pm 2\} \cap \{0,1,2,3,4\}$.

Suppose first that $\ani \in \{0,4\}$.  There are four pairs that yield $\la q_1,q_2\ra \simeq \Hyp$ (and hence give $\bw=\aw$ and thus $\anip = \ani$) whereas the twelve others give $\anip=2$, yielding
\begin{align*}
\Num{\oap}{\aw} &= 4\Num{\oap-2}{\aw} + 12\Num{\oap-2}{2} \\ &= 4(\frac18 4^{\oap-2} + (2-\ani)2^{\oap-4}) + 12(\frac18 4^{\oap-2})
= \frac18 4^{\oap} + (2-\ani)2^{\oap-2}.
\end{align*}
Next suppose that $\ani=2$.  Then for each $\bw \in \{\Hyp,\Quat\}$ we have that $\dim(\aw-\bw)=2$, so $\aw-\bw$ is represented by exactly two choices of pairs $(q_1, q_2)$.  Thus the remaining $12$ choices of $(q_1,q_2)$ correspond to $\bw$ such that $\anip=2$.  This yields
$$
\Num{\oap}{\aw} = 2\Num{\oap-2}{0}+12\Num{\oap-2}{2}+2\Num{\oap-2}{4} = 16(\frac18 4^{\oap-2}) = \frac184^{\oap},
$$
as required.  

Finally, suppose $\ani\in \{1,3\}$.    If $\bw=\aw$, then we must have $\ang{q_1,q_2}\simeq \Hyp$; this accounts for $4$ pairs $(q_1,q_2)$.  For each of the three other elements $\bw$ of the same anisotropic dimension as $\aw$, a quick calculation using Table~\ref{Table:Witt} and Lemma~\ref{L:invariant} yields that $\dim(\aw-\bw)=2$ and so $\aw-\bw$ is representable by exactly two choices of $(q_1,q_2)$.  This accounts for $3\times 2=6$ pairs. The remaining six choices of $(q_1,q_2)$ therefore yield $\bw$ such that $\anip\neq \ani$, so necessarily $\anip=4-\ani$.  We thus infer 
\begin{align*}
\Num{\oap}{\aw} &= 4\Num{\oap-2}{\aw} + 6\Num{\oap-2}{\ani} + 6\Num{\oap-2}{4-\ani} \\&= 10(\frac18 4^{\oap-2} + (2-\ani)2^{\oap-4}) + 6(\frac18 4^{\oap-2} + (\ani-2)2^{\oap-4}),
\end{align*}
and the formula follows.
\end{proof}

Now consider the case that each odd part of $\partt$ has multiplicity exactly two.

\begin{lemma} \label{L:2dcase}
Let $\aw\in \W_\ratk$ be such that $\dim(\aw)=\ani \in \{0,2,4\}$, and let $\bo\geq 0$.  
Then the number $\Numb{\bo}{\aw}$ of isometry classes of $\bo$-tuples of degree-two quadratic forms representing $\aw$ is
$$
\Numb{\bo}{\aw} = 
\begin{cases}
\lfloor \frac18 7^\bo \rceil +1 & \text{if $\bo$ is even and $\ani=0$};\\
\lfloor \frac18 7^\bo \rceil -1 & \text{if $\bo$ is odd and $\ani=4$};\\
\lfloor \frac18 7^\bo \rceil  & \text{otherwise}.
\end{cases}
$$
where $\lfloor \frac18 7^\bo \rceil = \frac18(7^\bo-(-1)^\bo)$ denotes the closest integer to $7^\bo/8$.
\end{lemma}

\begin{proof}
Set $\obp:=\bo$.  
We can write the formula as
$
\Numb{\obp}{\aw} = \frac18 (7^\obp-(-1)^\obp) + \ep_{\ani,\obp},
$ where 
\begin{equation} \label{epi}
\ep_{\ani,\obp} = \begin{cases} (-1)^\obp &\text{if $\ani=0$ and $\obp$ is even, or $\ani=4$ and $\obp$ is odd, and}\\
0 & \text{otherwise}. \end{cases}
\end{equation}
Notice that if $\ani\in \{0,4\}$ then $\ep_{\ani,\obp-1}+(-1)^\obp=\ep_{\ani,\obp}$, for all $\obp$.

When $\obp=0$, then $\Numb{0}{\cls{0}}=1$ and $\Numb{0}{\aw}=0$ for all $\aw\neq \cls{0}$ so the formula holds. 
Assume $\obp\geq 1$ and let us count the number of ways, up to isometry,  to construct a $\obp$-tuple of degree-two quadratic forms $\qtup_\obp = [q_1,\qtup_{\obp-1}]$ representing $\aw$. 
There are $7$ choices for the form $q_1$, of which $6$ are anisotropic.   
Set $\bw=\cls{\qtup_{\obp-1}}=\aw-\cls{q_1}$ and let $\anip=\dim(\bw)$. By induction $\Numb{\obp-1}{\bw}$ is an invariant of anisotropic dimension so we can set $\Numb{\obp-1}{\anip}:=\Numb{\obp-1}{\bw}$.

Suppose $\ani\in \{0,4\}$.  If $\cls{q_1}=\cls{0}$ then $\bw=\aw$ and $\anip=\ani$; otherwise, $\anip=\dim(\bw)=\dim(\aw - \cls{q_1})=2$.  Therefore we have
\begin{align*}
\Numb{\obp}{\aw} &= \Numb{\obp-1}{\aw} + 6\Numb{\obp-1}{2} \\
&= \frac18 (7^{\obp-1}-(-1)^{\obp-1}) + \ep_{\ani,\obp-1} + 6(\frac18 (7^{\obp-1}-(-1)^{\obp-1}))\\
&= \frac18 (7^\obp - (-1)^\obp) + (-1)^\obp+ \ep_{\ani,\obp-1} = \frac18 (7^\obp - (-1)^\obp)+\ep_{\ani,\obp}.
\end{align*}
On the other hand, if $\ani=2$, then $\bw=\Hyp$ if $\cls{q_1}=\aw$ and   $\bw=\Quat$ if
 $\cls{q_1} = \Quat-\aw$.  Each of the remaining five choices of $\cls{q_1}$ gives $\bw$ such that $\anip=2$.  This yields the final relation
\begin{align*}
\Numb{\obp}{\aw} &= \Numb{\obp-1}{\Hyp} + 5\Numb{\obp-1}{2} + \Numb{\obp-1}{\Quat}\\
&= 7(\frac18 (7^{\obp-1}-(-1)^{\obp-1})) + \ep_{0,\obp-1} + \ep_{4,\obp-1}.
\end{align*}
Since $\ep_{0,\obp-1} +\ep_{4,\obp-1} = (-1)^{\obp-1}$, the formula follows.
\end{proof}


\begin{theorem} \label{T:count}
Let $M$ be a vector space of dimension $n\geq 1$ and let $A$ be an index set.  Suppose $M$ is decomposed as a direct sum of nonzero subspaces $\oplus_{\alpha\in A}M_\alpha$; set $\ai = \vert \{\alpha \mid \dim(M_\alpha)=i\}\vert$.
Let $\aw\in \W_\ratk$ and let $\ani =\dim(\aw)$ denote its anisotropic dimension.  If $n-\ani$ is a nonnegative even integer, then the number of ways $\Numc{\ao}{\bo}{\co}{\aw}$ of assigning a nondegenerate quadratic form to each of the $\vert A \vert$ subspaces such that the sum is equivalent to $\aw$ in the Witt group is
$$
\Numc{\ao}{\bo}{\co}{\aw} = \begin{cases}
\frac18 4^{\ao} 7^{\bo} 8^{\co} & \text{if $\co\geq 1$};\\
\frac18 4^{\ao} 7^{\bo}+ (2-\ani)2^{\ao-2} & \text{if $\co=0$, $\ao\geq 1$};\\
\lfloor \frac18 7^{\bo}\rceil + \ep_{\ani,\bo} & \text{if $\ao=\co=0$},
 \end{cases}
$$
where $\ep_{\ani,\bo}$ was defined in \eqref{epi}.  
\end{theorem}

\begin{proof}
Set $\co=\sum_{i=3}^n\ai$.
First suppose that there exists some $\alpha \in A$ such that $\dim(M_\alpha)\geq 3$.  Then the number of choices of quadratic forms on $M'=\oplus_{\beta \neq \alpha}M_\beta$ is $4^{\ao}7^{\bo}8^{\co-1}$.  Given such a choice, let $\bw$ be its Witt class.  Then $\dim(\aw-\bw)$ has the same parity as $\dim(M_\alpha)$.  Since $\dim(M_\alpha)\geq 3$, each of the 8 possible choices of $\aw-\bw$ can be realized on $M_\alpha$.  The formula follows.

Now suppose that $\co=0$, so that $\dim(M_\alpha)\in \{1,2\}$ for all $\alpha \in A$.  If $\bo=0$ or $\ao=0$ then we apply Lemmas~\ref{L:1dcase} and \ref{L:2dcase}, respectively.  Otherwise, letting $\W_2$ denote the subgroup of the Witt group of all quadratic forms of even anisotropic dimension, we deduce that
\begin{align}\label{E:summands}
\Numc{\ao}{\bo}{0}{\aw} &= \sum_{\bw \in \W_2}\Num{\ao}{\aw-\bw}\Numb{\bo}{\bw} =\sum_{\bw\in \W_2}\Num{\ao}{\aw-\bw} \left(\left\lfloor \frac18 7^{\bo} \right\rceil + \ep_{\dim(\bw),\bo}\right)\notag\\
&= 4^{\ao}\left\lfloor \frac18 7^{\bo} \right\rceil +  \sum_{\bw\in \W_2}\Num{\ao}{\aw-\bw}\ep_{\dim(\bw),\bo}
\end{align}
where at this last step we have used that $\aw-\bw$ ranges over \emph{all} Witt classes of quadratic forms of dimension of the same parity as $\dim(\aw)$, and thus all $4^{\ao}$ possible $\ao$-tuples of degree-one quadratic forms.

When $\bo$ is even, $\ep_{\dim(\bw),\bo}$ is nonzero only when $\bw=\cls{0}$, in which case $\ep_{0,\bo}=1$, so the final summand is
$$
\ep_{0,\bo}\Num{\ao}{\aw} = \frac184^{\ao} + (2-\ani)2^{\ao-2} 
$$
whereas when $\bo$ is odd, the only nonzero factor is $\ep_{4,\bo}=-1$ and the term corresponding to $\bw=\Quat$ has $\dim (\aw-\bw) = 4-\ani$; this yields 
$$
\ep_{4,\bo}\Num{\ao}{\aw-\Quat} = (-1) \cdot \left(\frac184^{\ao} + \left( 2-(4-\ani)\right)2^{\ao-1}\right) = -\frac184^{\ao} + (2-\ani)2^{\ao-2}.
$$
Thus the final summand in \eqref{E:summands} is precisely $(-1)^{\bo}\frac184^{\ao}+ (2-\ani)2^{\ao-2}.$  Expanding $\left\lfloor \frac18 7^{\bo} \right\rceil$ as in Lemma~\ref{L:2dcase}, we obtain, for $\ao,\bo>0$ and $\co=0$, 
$$
\Numc{\ao}{\bo}{0}{\aw} = 4^{\ao} \frac18 (7^{\bo}-(-1)^{\bo})  + (-1)^{\bo} \frac184^{\ao} + (2-\ani)2^{\ao-2},
$$
as required.
\end{proof}

By Theorem~\ref{T:mainthm}, the $\Orth(q)$ orbits in $\mathcal{O}_\partt(\ratk)$ are in bijection with $$\Ppt:=\{ \qtup \in \Qtup_{\partt}\mid  \cls{\qtup} = \cls{q}\}.$$
Therefore rephrasing Theorem~\ref{T:count} gives the desired result.


\begin{corollary}
Let $q$ be a nondegenerate quadratic form on an $n$-dimensional space $V$ and let $\partt \in \Partt(n)$. Denote by $\mathcal{O}_\lambda$ the corresponding algebraic nilpotent adjoint orbit of $\algOq$.  Write $\ai$ for the number of odd parts of multiplicity exactly $i$ in $\partt$ and $\co:= \sum_{i=3}^n \ai$. 
Then the number of $\Orth(q)$-orbits in $\mathcal{O}_{\partt}(\ratk)$ is
$\vert \Ppt \vert = 
 \Numc{\ao}{\bo}{\co}{\cls{q}}.
$
\end{corollary}

Finally, let us formulate an algorithm, suggested by the proofs above, for enumerating the elements of the set $\Ppt$. 
Set $\ani=\dim(\cls{q})$ and $n=\deg(q)$.
\begin{algorithm}  Let $\partt \in \Partt(n)$.  Write $m_j$ for the multiplicity of part $j$ in $\partt$, and let $D$ be the set of odd parts in $\partt$.  Set $m=\sum_{j\in D} m_{j}$.  
\begin{description}
\item[Step 1] If $m < \ani$ then $\Ppt = \emptyset$.  If $m=0$ and $\ani=0$ then $P_{\partt,\Hyp} = \{\Hyp\}$. 
Otherwise:
\item[Step 2]  Define a subset $E$ of $D$ as follows.  If there is at least one part $j\in D$ with $m_j\geq 3$, let $E=\{j\}$.  If  $m <4$ then let $E=D$.
Otherwise, choose $E$ to satisfy $\sum_{j\in E}m_{j} = 3$ if $\ani$ is odd and $\sum_{j\in E}m_{j} = 4$ if $\ani$ is even.

\item[Step 3]  Generate the set $S$ of all tuples $[q_{j} \mid j \in D\setminus E]$ and the (small) set $T$ of all tuples $[q_{j}\mid j\in E]$ (with $\deg(q_j)=m_j$ for each $j$).  
\item[Step 4]  By construction, for each $\qtup_S \in S$, there exist one or more tuples $\qtup_T \in T$ such that $\cls{\qtup_S}+\cls{\qtup_T}=\cls{q}$; include each of the resulting tuples $[\qtup_S ,\qtup_T]$ in $\Ppt$.  
\end{description} 
In particular, $\mathcal{O}_\lambda$ has no $\ratk$-rational points in $\so(q)$ if and only if $m<\ani$.
\end{algorithm}

\section{Representatives for nilpotent orbits}\label{S:representatives}

In this section, we show how to generate from an element of $\Nilp(\cls{q},n)$ (or $\Nilps(n)$) an explicit representative of the corresponding rational nilpotent orbit of $G=\Orth(q)$ (or $G=\SO(q)$) on $\LieG=\mathfrak{so}(q)$.  We set our notation for $\LieG$ and for irreducible $\sltwok$-modules in Section~\ref{SS:notation}.  We present the strategy for the algorithm in Section~\ref{SS:strategy}, and provide the steps in Sections~\ref{S:hyp} to \ref{S:ve}. 


\subsection{Bases for $\LieG$ and for $\sltwok$-modules}\label{SS:notation}
Suppose $q \cong \Hyp^{\oplus m} \oplus \qani$ where the anisotropic kernel $\qani$ is represented by $\la r_1, r_2, \ldots, r_\ani \ra$; then $\dim(V)=n=2m+\ani$.  
Let $\{v_1, \ldots, v_m, w_1, \ldots, w_m\}$ be a Witt basis of $\Hyp^{\oplus m}$, that is, with $B_q(v_i, w_j) = \delta_{i,j}$, such that the subspace generated by the $v_i$s (respectively, the $w_i$s) is totally isotropic.  Complete this to a basis of $V$ by choosing, for $1\leq \kay,\ell\leq \ani$, vectors $z_\ell$ in the orthogonal complement such that $B_q(z_\ell,z_{\kay})=\delta_{\ell,\kay}r_\ell$.  Then with respect to the ordered basis $B=\{v_1, \ldots, v_m, w_1, \ldots, w_m, z_1, \ldots, z_{\ani}\}$ of $V$, and the corresponding dual basis $B^*=\{v_1^*, \ldots, v_m^*, w_1^*, \ldots, w_m^*, z_1^*, \ldots, z_\ani^*\}$ of $V^*$, the Lie algebra $\LieG \subset \gl(V)$ has a maximal split toral subalgebra $\LieT$ spanned by
$$
\Hi = v_iv_i^* - w_i w_i^* 
,\quad \text{for\:}  1\leq i \leq m.
$$
Then its centralizer $\LieG^{\LieT} = \LieT \oplus \gz$, where $\gz\cong \so(\qani)$ is spanned by $\{r_{\kay}z_\ell z_{\kay}^* - r_{\ell}z_{\kay}z_{\ell}^* \mid 1\leq \ell < \kay \leq \ani\}$.
Denoting by $\ep_i \in \LieT^*$ the functional $\ep_i(\Hj) = \delta_{i,j}$, for each $1\leq i, j \leq m$, the positive roots of $\LieG$ with respect to $\LieT$ are $\Phi^+=\{  \ep_i \pm \ep_j,  \ep_k \mid 1 \leq i < j \leq m, 1 \leq k \leq m\}$ and the root system is $\Phi = \Phi^+\cup (-\Phi^+)$.  A basis for each root space is given as follows:
\begin{align}
\ep_i-\ep_j  &\ (1 \leq i \neq j \leq m):& \Ximj &= v_iv_j^*-w_jw_i^* \notag \\
\ep_i+\ep_j  &\ (1 \leq i <  j\leq m):&\Xipj &= v_iw_j^* - v_jw_i^* \notag \\ 
-\ep_i-\ep_j  &\ (1 \leq i <  j\leq m):&\X_{-i,j} &=   w_jv_i^* - w_iv_j^* \label{E:sobasis}\\ 
\ep_i  &\ (1 \leq i \leq m):&\{\Xil  &= z_\ell w_i^* - r_\ell v_i z_\ell^* 
\mid 1 \leq \ell \leq \ani\} \notag\\
-\ep_i &\ (1 \leq i \leq m):& \{\X_{-i}^\ell &= z_\ell v_i^* -r_\ell w_i z_\ell^* 
\mid 1 \leq \ell \leq \ani\}. \notag  
\end{align}
With respect to these choices, we have the expected relations 
$[\Ximj,\X_{j,i}]=\Hi-\Hj$, 
$[\Xipj,\X_{-i,j}]=\Hi + \Hj$,
$[\Xil,\X_{-i}^{\kay}] =   \delta_{\ell,\kay}r_\ell \Hi$,
and also for $i<j$
$$
[\Xil,  \X_{j}^{\kay}] = -r_\ell \delta_{\ell,\kay}\Xipj, \quad
[\X_{-i}^\ell,  \X_{-j}^{\kay}] = r_\ell \delta_{\ell,\kay}\X_{-i,j}, \quad
[\Xil,  \X_{-j}^{\ell}] = r_\ell\Ximj,
$$
whereas if $\ell\neq \kay$ and $i\neq j$ we have $[\Xil,\X_{j}^{\kay}]\in \mathfrak{s}$.

Suppose now $\LieTrip=\{Y,H,X\}\subset \so(q)$ is a Lie triple.  
Let $\slmod_i$ be an $\sltwok$-submodule of $V$.  Then a basis for $\slmod_i$ is given by $\{X^{i-1}v, X^{i-2}v, \ldots, Xv, v\}$, where $v \in \slmod_i$ is a lowest weight vector; we'll call such an ordered basis an $\sltwok$-basis.   With respect to this basis, the action of $X$ is given in matrix form as a Jordan block $J_i$, that is, an upper triangular matrix with 1s on the second diagonal and 0s elsewhere.  In fact, $H,X,Y$ act by, respectively, the matrices
\begin{equation} \label{sl2form}
h_i = \diag(i-1, i-3, \cdots, -i+3, -i+1), \quad x_i = J_i, \quad \text{and}\quad y_i = D_i\;{{}^tJ_i}
\end{equation}
where $D_i = \diag(0,\mu_1, \cdots, \mu_{i-1})$ with $\mu_k = k(i-k)$ for $1\leq k <i$.  Importantly, by our hypotheses on $p$ we have that the residual characteristic satisfies $p>h$ (where the Coxeter number $h$ of $G$ is $h=2n$ if $n$ is odd, and $h=2(n-1)$ if $n$ is even) so that $\mu_k \in \R^\times$, regardless of $i$.  We use this property at various points, whose importance will be evident in Section~\ref{S:minimal}; its necessity for the DeBacker correspondence was discussed in \cite{Nevins2011a}.

We want to describe various $\sltwok$-submodules with respect to the basis $B$ of $V$ in order to construct our desired matrix representatives.  In Sections~\ref{S:hyp} to \ref{S:ve} we generally do so in one of two ways.

In the first way, given a consecutive subset of the Witt basis, which we take without loss of generality to be $B'=\{v_1,\ldots,v_i,w_1,\ldots, w_i\}$, we realize its span $V'$ as two copies of $U_i$ by choosing the $\sltwok$-bases $B_1 = \{v_1, v_2, \ldots, v_i\}$ and $B_2=\{-w_i, w_{i-1}, \ldots, (-1)^iw_1\}$.  Then
the restriction of $\LieTrip$ to $V'$ is given in matrix form by
\begin{equation}\label{E:hyperbolic}
H|_{V'} = \diag(h_i, -h_i), \quad 
X|_{V'} = \mat{x_i & 0\\ 0& -{}^tx_i}, \quad \text{and} \quad Y|_{V'} = \mat{y_i &0\\ 0& -{}^ty_i}.
\end{equation}

Alternatively, if $i = 2k+1$ is odd, we may take a consecutive subset such as $B_k = \{v_1, \ldots, v_k, w_1, \ldots, w_k\}$ of the Witt basis, together with a vector $x$ in the span of $B\setminus B_k$  (which is not necessarily in $\spn\{z_1, \cdots, z_{\ani}\}$!) satisfying $q(x)=r$, which will take the role of the 0-weight vector of the module.  Then
$$
B' = \{rv_1, rv_2, \ldots, rv_k, x, -w_k, w_{k-1}, \ldots, (-1)^kw_1\}
$$
is an $\sltwok$-basis of a submodule $V'$ isomorphic to $U_i$, such that $q|_{V'}\cong \la r \ra \oplus \Hyp^{\oplus k} \cong \la r \ra \otimes (\la 1 \ra \oplus \Hyp^{\oplus k})$.  The restriction of $\LieTrip$ to $V'$ is given in matrix form relative to the ordered basis $B_k\cup \{x\}$ by
\begin{align}\label{E:xbasis}
H|_{V'} &= \diag(\tilde{h}_k, -\tilde{h}_k, 0), \quad \text{where} \quad  \tilde{h}_k =  \diag(2k, 2k-2, \cdots, 2),\\ 
X|_{V'} &= \mat{x_k &  0& M_r \\ 0& -{}^tx_k & 0 \\ 0 & {}^tM_1 & 0}, \quad \text{and} \quad
Y|_{V'} = \mat{\tilde{y}_k &  0& 0 \\ 0& -{}^t\tilde{y}_k & M_{-1} \\ {}^tM_{r^{-1}} & 0 & 0}, \notag
\end{align}
where $M_s$ denotes the $k\times 1$ matrix $(\delta_{i,k}s)_{1\leq i\leq k}$, and $\tilde{y}_k$ is the $k$th order principal submatrix of $y_i$.   We make the convention that if $k=0$ then $\tilde{h}_{0}$ is omitted.

\subsection{The strategy} \label{SS:strategy}

Now suppose that $(\partt, \qtup)  \in \Nilp(\cls{q},n)$ or $\Nilps(n)$, so that $\partt \in \Partt(n)$ and 
$\qtup = [q_{i_1}, \cdots, q_{i_s}] \in \Ppt$ (or is taken to be $\Hyp$ if $\partt$ has no odd parts). Let $m_i$ be the multiplicity of $i$ in $\partt$.  
Thus each distinct part in $\partt$ corresponds to an orthogonal component of the direct sum \eqref{decompgen}, on which the restriction of $q$ is equivalent to $q_{i_k}$ if the part $i_k$ is odd, or else to a direct sum of hyperbolic planes if the part is even.  When $\noar$, some complications may arise, as in the following example.

\begin{example}
Suppose $\partt = (1,3,5)$ and $q =\langle 1\rangle \oplus \Hyp^{\oplus 4}$.   If $\scno$, then  
$\Ppt = \{ [1, a, -a], [a,1,-a], [a,-a,1] \mid a\in \{\pm 1, \pm \varpi\}\}$ (writing $a$ in place of $\la a \ra$) which has $\Num{3}{q} = 10$ distinct elements.  To $[1,\varpi, -\varpi]\in \Ppt$, for example, one associates a decomposition $V=V_1\oplus V_3\oplus V_5$ with quadratic forms $q_1=\la 1 \ra$, $q_3=\la \varpi \ra \oplus \Hyp$ and $q_5=\la -\varpi \ra \oplus \Hyp^{\oplus 2}$, respectively.  Thus $V_3\oplus V_5$ is isomorphic to a direct sum of hyperbolic planes (though individually neither one is), and $(q_1,V_1) \simeq (q,V)$.  On the other hand, however, now suppose \noar.  Then to $\qtup = [ \rho,\rho,\rho]\in \Ppt$ we associate a decomposition as above but in this case, no single $V_i$ carries a form equivalent to $q$!
\end{example}

As this example illustrates, the difficulty is that although $q \simeq q_{i_1}\oplus \cdots \oplus q_{i_s}$, it is not in general true that one can partition the basis $B$ to reflect this orthogonal decomposition, and thus choices of embeddings of $(q_i,V_i)$ into $(q,V)$ must be made. 
An optimal choice, as will be discussed in Section~\ref{S:minimal}, is one for which the span of the roots on whose spaces $X$ and $Y$ are supported has the smallest dimension. We proceed as follows.

We set some notation.  Recall that we have fixed a diagonal representative $\la r_1, \ldots, r_\ani\ra$ for $\qani$.  
Let $(\partt,\qtup) \in \Nilp(\cls{q},n)$ or $\Nilps(n)$.    Let $$\Ip=\{(i,j) \mid \text{$i \in \partt$, $1\leq j\leq m_i$}\}$$ be the set of all \emph{index pairs}, which has cardinality $\vert \partt \vert = \sum_{i\in \partt}m_i$, the number of parts in $\partt$.  Write $\qtup=[q_{j_1},\ldots, q_{j_s}]$ and for each odd part $i=j_t$ ($1\leq t \leq s)$ of $\partt$ let $q_i\simeq \la r_{i,1}, r_{i,2}, \ldots, r_{i,m_i}\ra$ be a diagonal form of $q_i$. This defines a choice of map $(i,j)\mapsto r_{i,j}$ on all $(i,j)\in \Ip$ such that $i$ is odd.  

\begin{proposition}\label{P:partition}
There exists a partition $\Gamma$ of $\Ip$ into subsets of the form: 
\begin{description}
\item[\labeleven] $\{(i,j), (i,j+1)\}$ such that $i$ is even; 
\item[\labelhyp] $\{(i,j), (i,j')\}$ such that $i$ is odd, $j\neq j'$ and $\la r_{i,j}\ra \cong \la -r_{i,j'}\ra$;
\item[\labelpairs] $\{(i,j), (i',j')\}$ such that $i\neq i'$ are odd, and $\la r_{i,j}\ra\cong\la -r_{i',j'}\ra$; 
\item[\labelfours] if \noar: $\{(i_1,j_1), (i_2,j_2), (i_3,j_3), (i_4,j_4)\}$ such that the $i_k$ are distinct odd parts and $\la r_{i_1,j_1}\ra\cong\la r_{i_2,j_2}\ra\cong\la r_{i_3,j_3}\ra\cong\la r_{i_4,j_4}\ra$; 
\item[\labeltriple] if \noar, $\{(i_1,j_1), (i_2,j_2), (i_3,j_3)\}$ such that the $i_k$ are distinct odd parts, and there exists $1 \leq \ell \leq \ani$ such that $\la-r_\ell\ra\cong\la r_{i_1,j_1}\ra\cong\la r_{i_2,j_2}\ra\cong\la r_{i_3,j_3}\ra$;
\item[\labelsign] if \noar, $\{(i,j), (i',j')\}$ such that $i\neq i'$ are odd and there exist $1\leq \kay <\ell \leq \ani$ such that $\la r_{i,j}\ra\cong\la r_{i',j'}\ra\cong\la -r_\kay \ra \cong \la -r_\ell \ra$;
\item[\labelani] for some $t\leq \ani$: $\{(i_s, j_s) \mid 1\leq s \leq t\}$ such that up to permutation of the diagonal representative of $\qani$ we have $\la r_{i_s,j_s}\ra\cong\la  r_{s}\ra$ for each $1\leq s \leq t$;
\end{description}
and such that there are in total at most two sets of the form \labeltriple, \labelsign\ or \labelani\ in $\Gamma$.  
\end{proposition}

\begin{proof}
Even parts occur with even multiplicity, so the index pairs $(i,j)$ with $i$ even can be partitioned into couples of the form \labeleven.  Thus it suffices to consider the case that all parts $i$ of $\lambda$ are odd.

If \scno, then $\la r \ra = \la -r\ra$ so exhaustively matching up elements of $\Ip$ using \labelhyp\ or \labelpairs\ leaves at most four index pairs $(i,j)$, such that each $r_{i,j}$ represents a distinct square class in $\{1,\rho,\varpi, \rho\varpi\}$.  Since $\cls{\qtup}=\cls{q}$, we deduce these must satisfy \labelani.  Therefore the lemma holds in this case.

If \noar, then exhaustively matching up elements using \labelhyp\ or \labelpairs\ leaves at most two sets of index pairs:  $S_a$ and $S_{\varpi b}$, consisting of those $(i,j)$ for which $\la r_{i,j}\ra\cong\la a\ra$ or $\la r_{i,j}\ra \cong \la \varpi b \ra$, for some fixed $a,b \in \{1,\rho\}$, respectively.   We claim that $S_a$ and $S_{\varpi b}$ can each be partitioned as required, with at most one part of the form \labeltriple, \labelsign\ or \labelani. 

If $S_a \cup S_{\varpi b}$ satisfies \labelani, we are done.  If not, then given the classification of Table~\ref{Table:Witt}, there is at least one $r \in \{a,\varpi b\}$ such that $\vert S_r \vert >2$; we now consider each such $r$ in turn.  

If $S_r$ contains two distinct elements $(i,j)$ and $(i,j')$ with the same part $i$ of $\partt$, then by applying the relation $\la r, r \ra \cong \la -r, -r \ra$, we can replace the diagonal form of the corresponding $q_i$ with one in which $r_{i,j}=r_{i,j'} = -r$ instead.   Since $\vert S_r\vert >2$, we may then apply \labelhyp\ or \labelpairs, removing at least two (and perhaps four) elements from $S_r$.  Repeat this process until either: the resulting $S_r$ has two or fewer elements, in which case \labelani\ applies and we are done; or all the parts $i$ of $\partt$ occuring in elements of the resulting $S_r$ are distinct.  In the latter case, we next exhaustively eliminate quadruples from $S_r$ using \labelfours,  leaving at most three elements.  Recalling that $\la r,r,r\ra \simeq \la -r\ra$ and $\la r,r\ra \cong \la -r,-r\ra$, we conclude that exactly one of \labeltriple, \labelsign\ or \labelani\ must apply to what is left of $S_r$, and we are done.  
\end{proof}

Choose such a partition $\Gamma$.  For each $\lab \in \{\labeleven, \labelhyp, \labelpairs, \labelfours, \labeltriple, \labelsign, \labelani\}$, let $\Gamma_\lab$ denote the set of parts of $\Gamma$ falling under case $\lab$.  For each $\gamma \in \Gamma_\labeltriple \cup \Gamma_\labelsign \cup \Gamma_\labelani$, if any, we have that $\oplus_{(i,j)\in \gamma}\la r_{i,j}\ra$ represents a nontrivial element of the Witt group, and these are the only such parts in $\Gamma$.

In the following sections we partition the basis $B$ according to $\Gamma$,
giving an orthogonal decomposition $V = \oplus_{\gamma \in \Gamma}V_\gamma$.
 On each orthogonal subspace $V_\gamma$ we construct an action of $\sltwok$ such that 
 their direct sum is isomorphic to a decomposition \eqref{decompgen} associated to $(\lambda, \qtup)$.  In doing so, we define a Lie triple $\LieTrip = \{Y,H,X\}$ such that $X$ represents the nilpotent orbit corresponding to $(\lambda, \qtup)$.   Since $X$ is a sum of its restrictions to each $V_\gamma$, and each $V_\gamma$ is spanned by a subset of $B$, we will recoup an expression for $X$ as a linear combination of root vectors, and thus an explicit representative of the orbit.

For ease of notation, in each case we suppose the subset $B_\gamma$ of $B$ starts with $\{v_1, w_1\}$; to implement this algorithm in practice, one chooses an appropriate partition of $B$ and shifts the indices on all the vectors and root vectors.

\subsection{$\Gamma_\labeleven$ : even parts of $\partt$, and $\Gamma_\labelhyp$ : hyperbolic planes in odd parts of $\partt$} \label{S:hyp}
Suppose $\gamma=\{(i,j),(i,j')\} \in \Gamma_\labeleven \cup \Gamma_\labelhyp$.
Then $V_\gamma$ should be a sum of two copies of $\slmod_i$ and the restriction of $q$ to $V_\gamma$ should be a split quadratic space of dimension $2i$.  Therefore we choose a consecutive subset of the Witt basis with $2i$ elements, which up to relabeling we may take to be
$B_\gamma=\{v_1, \ldots, v_{i}, w_1, \ldots, w_{i}\}$. 
Set $V_\gamma = \spn(B_\gamma)$ and define an action of $\LieTrip$ as in \eqref{E:hyperbolic}.  Then $V_\gamma \cong \slmod_i\oplus \slmod_i$ and the restriction of $q$ to $V_\gamma$ is isometric to $\Hyp^{\oplus i}$. 
From the matrix form we deduce that $X|_{V_\gamma}$ is a sum of simple root vectors. Specifically, we have
\begin{equation}\label{E:hypX}
X|_{V_\gamma} = \sum_{1 \leq t < i} \X_{t,t+1}, \quad \text{and}\quad Y|_{V_\gamma}=\sum_{1 \leq t < i} \mu_t\X_{t+1,t},
\end{equation}
 where we note that $Y|_{V_\gamma}$ is a linear combination of exactly the corresponding negative root vectors.

By choosing disjoint subsets $B_\gamma$ for each $\gamma \in \Gamma_\labeleven \cup \Gamma_\labelhyp$, we thus create an $\sltwok$-submodule  $V_{\labelhyp}=\oplus V_\gamma$ of $V$.

\subsection{$\Gamma_\labelpairs$: Hyperbolic planes across two distinct odd parts} \label{S:pairs}
Suppose now that $\gamma = \{(i,j),(i',j')\} \in \Gamma_\labelpairs$. Scaling by a square if necessary, we may assume that $r_{i,j}=r$ and $r_{i',j'}=-r$ for some $r$.  Write $i=2k+1$, $i'=2k'+1$, with $k>k'\geq 0$, and set $p = \frac12(i+i')=k+k'+1$.  

Here, $V_\gamma$ should be isomorphic to $\slmod_i \oplus \slmod_{i'}$ as $\sltwok$-modules, and the restriction of $q$ to $V_\gamma$ should be $(\la r \ra \otimes (\la 1 \ra \oplus \Hyp^{\oplus k}))\oplus (\la -r \ra \otimes (\la 1 \ra  \oplus \Hyp^{\oplus k'})) \cong \Hyp^{\oplus 2p}$.  We use the relation 
\begin{equation}\label{E:hyptodiag}
{}^tK\mat{0&1\\1&0}K = \mat{r&0\\0&-r}, \quad \text{with}\; K=\mat{r/2 & r/2\\ 1&-1},
\end{equation}
to explicitly identify the two-dimensional 0-weight space, which carries the form $\langle r, -r\rangle$, with a hyperbolic plane $\Hyp$ spanned by a Witt basis.

That is, choose a subset of the Witt basis with $2p$ elements, which up to relabeling we take to be
$B_\gamma=\{v_1, \ldots, v_{p}, w_1, \ldots, w_{p}\}$.
Identify the 2-dimensional 0-weight space with the span of $\{v_{p}, w_{p}\}$ by choosing 
$x^\pm_p=\frac{r}{2}v_{p}\pm w_{p}$ as our $0$-weight vectors.  Note that $q(x^\pm_p)=\pm r$.  Then $\sltwok$-bases for the corresponding decomposition of $B_\gamma$ into orthogonal $\SL_2(k)$-invariant subspaces are
\begin{equation}\label{E:basischoice}
B_{i}= \{rv_1, rv_2, \ldots, rv_k, x_p^+, -w_k, w_{k-1}, \ldots, (-1)^{k}w_1\}
\end{equation}
and
$$
B_{i'} =\{ rv_{k+1}, rv_{k+2}, \ldots, rv_{k+k'}, x^-_p, 
w_{k+k'}, -w_{k+k'-1}, \ldots, (-1)^{k'+1}w_{k+1} \},
$$
respectively, giving the required $\sltwok$-structure to $V_\gamma = \spn{(B_\gamma)}$.

We next write down explicit representatives of the restrictions of $H$ and of $X$ to $V_\gamma$ with respect to $B_\gamma$, which amounts to performing a change of coordinates from $B_i\cup B_{i'}$, with respect to which these matrices are given as in \eqref{sl2form}.   

We have $H|_{V_\gamma} = \diag(\tilde{h}_k, \tilde{h}_{k'}, 0, -\tilde{h}_k, -\tilde{h}_{k'}, 0)$.  
If $k'=0$ then it follows from the bases above that $Xv_1=0$, $Xv_p = v_k$, and $Xv_i = v_{i-1}$ for $1<i<p$, whereas if $k'>0$ then instead $Xv_{k+1}=0$ and $Xv_{p} = v_k + v_{k+k'}$.  Similarly, if $k'=0$ we have $Xw_p=\frac{r}{2}v_k$, $Xw_k = -\frac{r}{2}v_p-w_p$ and $Xw_i = -w_{i+1}$ for all $1\leq i<k$, whereas if $k'>0$ then instead $Xw_p = \frac{r}{2}v_k-\frac{r}{2}v_{k+k'}$ and $Xw_{k+k'} = \frac{r}{2}v_p-w_p$.
Using the notation of \eqref{E:sobasis}, we may thus write the restriction of $X$ to $V_\gamma$ as the sum of positive root vectors
\begin{equation}\label{E:pairsX}
X|_{V_\gamma} = \sum_{\substack{1\leq j < k+k'\\ j\neq k} }\X_{j,j+1} + \X_{k,p}  +  \frac{r}{2}\X_{k,-p} 
+ \X_{k+k',p}
- \frac{r}{2}\X_{k+k',-p}
\end{equation}
where if $k'=0$ we omit the two  terms in which $k+k'$ appears as a subscript.  Similarly,  $Y|_{V_{\gamma}}$ is a linear combination (with coefficients in $\R^\times$) of the root vectors $$\{\X_{j+1,j}, \X_{p,k}, r^{-1}\X_{-k,p}, (\X_{p,k+k'}), (r^{-1}\X_{-p,k+k'})\mid 1\leq j <k+k', j\neq k\},$$
omitting the terms in parentheses when $k'=0$.

Making suitable choices of disjoint bases $B_\gamma$, for each $\gamma \in \Gamma_\labelpairs$, yields another split quadratic subspace $V_{\labelpairs}=\oplus_{\gamma \in \Gamma_\labelpairs}V_\gamma\subseteq V$.  

\subsection{$\Gamma_{\labelfours}$: Hyperbolic planes across four parts, when $\noar$} \label{S:fours}

Suppose now that $$\gamma=\{(i_1,j_1), (i_2,j_2), (i_3,j_3), (i_4,j_4)\} \in \Gamma_\labelfours$$ with $i_1> i_2> i_3> i_4$, and after scaling by squares if necessary, let $r$ be the common value of $r_{i_t,j_t}$ for $1\leq t\leq 4$.  Let $i_t=2k_t+1$ for each $t$, and set $p=\frac12\sum i_t =k_1+k_2+k_3+k_4+2$.  Choose a subset of the Witt basis with $2p$ elements, which we assume up to relabelling is $
B_\gamma= \{v_1, \cdots, v_p, w_1, \cdots, w_p\}.
$
This space is to carry the module $\oplus_t U_{i_t}$ with form $\langle r,r,r,r \rangle\oplus \Hyp^{\oplus p-2}$.  We choose its hyperbolic four dimensional $0$-weight space  to coincide with $W_0=\spn\{v_{p-1},v_p, w_{p-1},w_p\}$.   Using the change of basis matrix $K$ from \eqref{E:hyptodiag} we diagonalize $B_q$ on this subspace to $\diag(r,r,-r,-r)$.  Next, since $\noar$, there exist $c,s\in \ratk^\times$ such that $c^2+s^2=-1$.  A matrix $C$ satisfying ${}^tC (-rI) C = rI$ is given by
\begin{equation}\label{E:Cmatrix}
C = \mat{c&-s\\ s&c}.
\end{equation}
Consequently, the following vectors form an orthogonal basis of $W_0$ in which each vector $x_t$ satisfies $q(x_t)=r$:
\begin{align}\label{E:changebasismatrix}
x_1&=\frac{r}{2}v_{p-1}+w_{p-1}, \quad \\ \notag
x_2&=\frac{r}{2}v_{p}+w_{p}, \\ \notag
x_3&=\frac{cr}{2}v_{p-1}+\frac{sr}{2}v_p - cw_{p-1}-sw_p, \\ \notag
x_4&=\frac{-sr}{2}v_{p-1}+\frac{cr}{2}v_p+sw_{p-1}-cw_{p}. 
\end{align}
We complete each of these to an $\sltwok$-basis of $U_{i_t}$, respectively, by partitioning the remaining elements of $B_\gamma$ as before.  Specifically, setting $p_t = \sum_{s=1}^{t}k_s$, so that $p_0=0$ and $p_4=p-2$, ordered bases of the four $\sltwok$ submodules are
\begin{equation}\label{E:hypbasis2}
B_t = \{ rv_{p_{t-1}+1}, rv_{p_{t-1}+2}, \cdots, rv_{p_t}, x_{t}, -w_{p_t}, \cdots, (-1)^{k_t}w_{p_{t-1}+1} \},
\end{equation}
where it is understood that if $k_4=0$ then $B_4 = \{x_4\}$, a one-dimensional space.

These bases define the restriction of the Lie triple $\{Y,H,X\}$ to $V_\gamma=\spn(B_\gamma)$.  For example, the matrix of $H_{V_\gamma}$ is $$\diag(\tilde{h}_{k_1}, \tilde{h}_{k_2}, \tilde{h}_{k_3}, \tilde{h}_{k_4},0,0,-\tilde{h}_{k_1},-\tilde{h}_{k_2},-\tilde{h}_{k_3},-\tilde{h}_{k_4},0,0).$$
To obtain the matrix of the restriction of $X$ to $V_\gamma$ with respect to $B_\gamma$, we first invert \eqref{E:changebasismatrix}, and then apply the relations $Xx_i = rv_{p_i}$ to deduce
\begin{gather*} 
Xv_{p-1} = v_{p_1} -cv_{p_3} + sv_{p_4}\\
Xv_{p} = v_{p_2} -sv_{p_3} - cv_{p_4}\\
Xw_{p-1} = \frac{r}{2}v_{p_1} + \frac{cr}{2}v_{p_3} - \frac{sr}{2}v_{p_4}\\
Xw_{p} = \frac{r}{2}v_{p_2} + \frac{sr}{2}v_{p_3} + \frac{cr}{2}v_{p_4},
\end{gather*}
where if $k_4=0$ we omit the four terms containing the subscript $p_4$.
The action of $X$ on the remaining $v_i$ and $w_i$ of $B_\gamma$ can be read from the bases \eqref{E:hypbasis2} directly, and contribute sums of simple root vectors as before.  With respect to the root vectors \eqref{E:sobasis}, the restriction of $X$ to $V_\gamma$ is given by
\begin{align}\label{E:quadX}
X|_{V_\gamma} = \sum_{1\leq j <p-2, j\neq p_1,p_2,p_3} &\X_{j,j+1}  
+ \X_{p_1,p-1} + \X_{p_2,p}  - c\X_{p_3,p-1} -s \X_{p_3,p} \notag \\
&+ \frac{r}{2}\X_{p_1,-(p-1)} + \frac{r}{2} \X_{p_2,-p} + 
\frac{cr}{2}\X_{p_3,-(p-1)} +\frac{sr}{2} \X_{p_3,-p} \notag\\
&+ s\X_{p_4,p-1} -c\X_{p_4,p} +\frac{-sr}{2}\X_{p_4,-(p-1)} + \frac{cr}{2}\X_{p_4,-p},
\end{align}
where the four terms containing $p_4$ as a subscript are omitted if $k_4=0$.  As before, one can verify that $Y|_{V_\gamma}$ is a linear combination of the corresponding negative root vectors, with the proviso that if a root vector appears with coefficient in $a\R^\times$ in \eqref{E:quadX} for some $a\in \ratk$ then the corresponding negative root vector appears with a coefficient in $a^{-1}\R^\times$ in $Y|_{V_\gamma}$, that is, with the negative valuation.

Choosing disjoint subsets $B_\gamma$, for each $\gamma \in \Gamma_\labelfours$, gives an $\sltwok$-invariant split quadratic subspace $V_{\labelfours}=\oplus_{\gamma\in \Gamma_\labelfours} V_\gamma$ of $V$.

\subsection{$\Gamma_\labeltriple$: Anisotropic part, when $\noar$ and a triple identity is required} \label{S:noarani2}
Suppose  $\noar$ and $\gamma =\{(i_1,j_1), (i_2,j_2), (i_3,j_3)\}\in \Gamma_\labeltriple$; without loss of generality we assume $i_1>i_2>i_3$ and $r_{i_1,j_1}=r_{i_2,j_2}=r_{i_3,j_3}=-r$. Let $\ell\in \{1,\ldots, \ani\}$ be such that $r=r_\ell$; then by the proof of Proposition~\ref{P:partition} we know that this is the only occurrence of $r$ (up to scaling by $\ratks$) in the diagonal form of $\qani$.  Let $k_t=(i_t-1)/2$ for each $t$ and set $p=k_1+k_2+k_3+1$.   In this case, $V_\gamma$ should be isomorphic to $\slmod_{i_1}\oplus \slmod_{i_2} \oplus \slmod_{i_3}$ and carry the form $\la -r,-r,-r \ra \oplus \Hyp^{\oplus p-1} \cong \la r \ra \oplus \Hyp^{\oplus p}$.
Therefore, up to relabeling of the Witt basis, we choose the subset
$$
B_\gamma = \{v_1, \cdots, v_{p}, w_1, \cdots, w_{p}, z_\ell\}.
$$
The restriction of $q$ to $W_0=\spn\{v_p,w_p,z_\ell\}$ can be transformed to the diagonal form $\la -r,-r,-r\ra$ by first applying the matrix $K$ of \eqref{E:hyptodiag} to the first two coordinates, then $C$ of \eqref{E:Cmatrix} to the first and last coordinates.  Thus the orthogonal vectors
$$
x_1 = \frac{cr}{2}v_p + cw_p+sz_\ell, \quad x_2 = \frac{r}{2}v_p - w_p, \quad \text{and}
\quad x_3 = -\frac{sr}{2}v_p-sw_p+cz_\ell
$$
each satisfy $q(x_i)=-r$, which implies that the following bases span complementary quadratic subspaces of $V_\gamma=\spn(B_\gamma)$, each with anisotropic kernel $\la -r \ra$:
\begin{gather*}
B_1 = \{ -rv_1, \ldots, -rv_{k_1}, x_1, -w_{k_1}, \ldots, (-1)^{k_1}w_1\}\\
B_2 = \{ -rv_{k_1+1}, \ldots, -rv_{k_1+k_2}, x_2, -w_{k_1+k_2}, \ldots, (-1)^{k_2}w_{k_1+1}\}\\
B_3 = \{ -rv_{k_1+k_2+1}, \ldots, -rv_{k_1+k_2+k_3}, x_3, -w_{k_1+k_2+k_3}, \ldots, (-1)^{k_3}w_{k_1+k_2+1}\}
\end{gather*}
where it is understood that $B_3 = \{x_3\}$ if $k_3=0$.
We interpret the $B_i$ as standard bases for $\sltwok$-modules as usual.  With respect to $B_\gamma$ the matrix of the restriction of $H$ to $V_\gamma$ is $\diag(\tilde{h}_{k_1},\tilde{h}_{k_2},\tilde{h}_{k_3},0, -\tilde{h}_{k_1},-\tilde{h}_{k_2},-\tilde{h}_{k_3},0,0)$.  The action of $X$ can be determined from the bases $B_t$ above, noting that
$$
v_p = -cr^{-1}x_1+r^{-1}x_2+sr^{-1}x_3, \quad w_p=\frac12(-cx_1-x_2+sx_3), \quad z_\ell=-sx_1-cx_3.
$$
It follows that in terms of the root vectors \eqref{E:sobasis} we have
\begin{align} \label{E:tripX}
X|_{V_\gamma} = \sum_{\substack{1\leq j < p-1\\ j\neq k_1, k_1+k_2} } &\X_{j,j+1}
+c\X_{k_1,p} - \X_{k_1+k_2,p} -s\X_{p-1,p}+\notag\\
&+\frac{cr}{2}\X_{k_1,-p} + \frac{r}{2}\X_{k_1+k_2,-p} - \frac{sr}{2}\X_{p-1,-p}
-s\X^\ell_{k_1} - c\X^\ell_{p-1}
\end{align}
where if $k_3=0$ we omit the three terms having $p-1$ as a subscript.  Similarly, we can readily determine that $Y|_{V_\gamma}$ is a linear combination (with coefficients in $\R^\times$) of the root vectors
\begin{gather*}
\{\X_{j+1,j}, \X_{p,k_1}, \X_{p,k_1+k_2}, (\X_{p,p-1}), r^{-1}X_{-k_1,p}, r^{-1}\X_{-(k_1+k_2),p}, (r^{-1}\X_{-(p-1),p}),\\ r^{-1}\X^\ell_{-k_1}, (r^{-1}\X^\ell_{p-1}) \mid 1\leq j <p-1, j\neq k_1, j\neq k_1+k_2\}
\end{gather*}
where the vectors in parentheses are omitted if $k_3=0$.

\subsection{$\Gamma_{\labelsign}$: Anisotropic part, when $\noar$ and sign change is required} \label{S:noarani1}

Now suppose  $\noar$ and $\gamma=\{(i,j), (i',j')\}\in \Gamma_{\labelsign}$.  We assume without loss of generality that $i>i'$ and $r_{i,j}=r_{i',j'}=-r$.   By the proof of Proposition~\ref{P:partition}, there are exactly two indices $\ell <\kay$ in $\{1,\ldots, \ani\}$ such that $r=r_\ell=r_\kay$.  Let $k=(i-1)/2$, $k'=(i'-1)/2$ and $p=k+k'$.  Then $V_\gamma$ should be isomorphic to $\slmod_i \oplus \slmod_{i'}$ as $\sltwok$-modules and carry the form $\la -r,-r \ra \oplus \Hyp^{\oplus p}$.  
Therefore, up to numbering of the Witt basis, we choose the subset
$$
B_\gamma = \{v_1, \cdots, v_{p}, w_1, \cdots, w_{p}, z_\ell, z_{\kay}\}.
$$
The matrix $C$  of \eqref{E:Cmatrix} transforms $\la r,r\ra$ to $\la -r,-r\ra$, so the vectors
$$
x_1 = cz_\ell-sz_{\kay} \quad \text{and}\quad x_2 = sz_\ell + cz_{\kay}
$$
each satisfy $q(x_i)=-r$.  Thus the following bases span complementary quadratic subspaces of
the span $V_\gamma$ of $B_\gamma$, each with anisotropic kernel $\la -r \ra$:
\begin{gather*}
B_1= \{ -rv_1, \ldots, -rv_{k}, x_1, -w_{k}, \ldots, (-1)^{k}w_1\}\\
B_2 = \{ -rv_{k+1}, \ldots, -rv_{k+k'}, x_2, -w_{k+k'}, \ldots, (-1)^{k'}w_{k+1}\}
\end{gather*}
where it is understood that $B_{2} = \{x_2\}$ if $k'=0$.
We interpret the $B_i$ as standard bases for $\sltwok$-modules as usual.  With respect to $B_\gamma$ the matrix of the restriction of $H$ to $V_\gamma$ is 
$\diag(\tilde{h}_{k},\tilde{h}_{k'}, -\tilde{h}_{k}, -\tilde{h}_{k'}, 0,0)$.  The action of $X$ can be read from the bases $B_t$ above, noting that
$$
z_\ell = -cx_1-sx_2\quad \text{and}\quad  z_\kay = sx_1-cx_2. 
$$
In terms of the root vectors \eqref{E:sobasis} we have
\begin{equation}\label{E:signX}
X|_{V_\gamma} = \sum_{1\leq j< k+k', j\neq k} \X_{j,j+1} -c\X_{k}^\ell +s\X_{k}^{\kay} - s\X_{k+k'}^{\ell} -c\X_{k+k'}^{\kay},
\end{equation}
where we omit the two terms with $k+k'$ in the subscript if $k'=0$.  On the other hand, $Y|_{V_\gamma}$ is a linear combination  (with coefficients in $\R^\times$) of the root vectors in the set
$$
\{ \X_{j+1,j}, r^{-1}\X^\ell_{-k}, r^{-1}\X^\kay_{-k}, (r^{-1}\X^\ell_{-(k+k')}), 
(r^{-1}\X^\kay_{-(k+k')}) \mid 1\leq j<k+k', j\neq k\}
$$
where the vectors in parentheses are omitted if $k'=0$.

\subsection{$\Gamma_{\labelani}$: Anisotropic part, simple case} \label{S:scnoani}
Suppose now that $\gamma \in \Gamma_\labelani$ and match each element $(i,j)$ of $\gamma$ to a distinct index $\ell \in \{1, \ldots, \ani\}$ such that $r_{i,j}=r_{\ell}$.  
Now fix $(i,j)\in \gamma$ and the corresponding index $\ell$.  Let $k=(i-1)/2$; then up to renumbering the elements of the Witt basis, choose the subset
$
B_{(i,j)} = \{ v_1, \ldots, v_k, w_1, \ldots, w_k, z_\ell \}
$
and denote its span $V_{(i,j)}$.  We rescale and reorder this basis to obtain the $\sltwok$-basis
$$
B_\ell = \{r_\ell v_1, \cdots, r_\ell v_k, z_\ell, -w_k, w_{k-1}, \cdots, (-1)^kw_1\}
$$
that spans an irreducible $\sltwok$-module isomorphic to $\slmod_i$ and carrying the form $\la r_\ell \ra \otimes (\la 1 \ra \oplus \Hyp^{\oplus k})$.
The restriction of $H$ to $V_{(i,j)}$ is given in matrix form by \eqref{E:xbasis}, with $x=z_\ell$.  In terms of our chosen scaling of root vectors in \eqref{E:sobasis} we have, if $k>0$,
\begin{equation}\label{E:aniX}
X|_{V_{(i,j)}} = \sum_{j=1}^{k-1}\X_{j,j+1} - \X_k^\ell, \quad Y|_{V_{(i,j)}} = \sum_{j=1}^{k-1}\mu_j\X_{j+1,j} + \mu_k r^{-1} \X_{-k}^\ell,
\end{equation}
and both are $0$ if $k=0$.  
Let $V_{\gamma}=\oplus_{(i,j)\in \gamma} V_{(i,j)}$, obtained by choosing suitable disjoint subsets $B_{(i,j)}$ of the basis $B$.

By Proposition~\ref{P:partition}, $\Gamma$ contains at most two parts corresponding to the cases \labeltriple, \labelsign, or \labelani, and their union corresponds to a subspace of $V$ carrying a form that is Witt-equivalent to $q$.  Let $V_\labelani = \oplus_{\gamma \in \Gamma_\labeltriple \cup \Gamma_\labelsign \cup \Gamma_\labelani}V_\gamma$ denote this subspace.

Putting all of the preceding constructions together, we have chosen a partition of the basis $B$ and a corresponding direct sum decomposition
\begin{equation}\label{E:sumall}
V = V_\labelhyp \oplus V_\labelpairs \oplus V_\labelfours \oplus V_\labelani
\end{equation}
together with an action of  a Lie triple $\{Y,H,X\}$, such that $X$ represents the nilpotent orbit of $\Orth(q)$ attached to $(\partt, \qtup)$.

\subsection{Very even orbits}\label{S:ve}
Now suppose that $\partt \in \Partte(n)$ is a very even partition.  A representative for the 
corresponding nilpotent $\Orth(q)$ orbit on $\so(q)$ was constructed in Section~\ref{S:hyp}, by pairing up irreducible $\sltwok$-submodules in the obvious way.
In this section, we modify one component in order to construct a second representative, such that the nilpotent $\SO(q)$ orbits of the two representatives are distinct.

Choose one element of $\gamma =\{(i,1),(i,2)\}\in \Gamma_\labeleven = \Gamma$, and up to relabelling let
$
B_\gamma = \{v_1, \ldots, v_i, w_1, \ldots, w_i\}
$
be the corresponding subset of the Witt basis.  Set $V_\gamma=\spn(B_\gamma)$.  This time, using the strategy of the proof of Theorem~\ref{T:mainthm} we first apply the orthogonal transformation of determinant $-1$ which permutes $v_i$ and $w_i$, to define
\begin{gather*}
B_1 = \{v_1, v_2, \cdots, v_{i-2}, v_{i-1}, w_i\}\\
B_2 = \{ v_i, -w_{i-1}, w_{i-2}, \cdots, (-1)^{i-1}w_1\}
\end{gather*}
as the $\sltwok$-bases for the two submodules isomorphic to $\slmod_i$.  With respect to $B_\gamma$, the restriction of $H$ to $V_\gamma$ now has the form $\diag(i-1, h_{i-2}, i-1, -i+1, -h_{i-2}, -i+1)$.  The action of $X$ and $Y$ on $V_\gamma$ can be read directly from $B_1$ and $B_2$; as a sum of root vectors, this yields
\begin{equation}\label{E:veryevenX}
X|_{V_\gamma} = \sum_{j=1}^{i-2}\X_{j,j+1} + \X_{i-1,-i}, \quad Y|_{V_\gamma} = \sum_{j=1}^{i-2}\mu_j\X_{j+1,j} + \mu_1\X_{-(i-1),i}.
\end{equation}
Putting this component in the place of $V_\gamma$ in \eqref{E:sumall}, we obtain a representative of the second $\SO(q)$-orbit corresponding to $\partt$.

\section{On functoriality and the DeBacker parametrization of nilpotent orbits}\label{S:functoriality}

Let us return to the more general setting of Section~\ref{SS:centralizers}, and adopt the notation introduced there.  


In \cite{DeBacker2002}, DeBacker gives a parametrization of rational nilpotent orbits of $G$ on $\LieG$ in terms of objects arising from its Bruhat-Tits building $\buil(G) = \buil(\algG,\ratk)$.  
To describe it, recall that to each $x\in \buil(G)$ Bruhat-Tits theory associates an $\R$-lattice $\LieG_{x,0}$; this is carefully described (particularly for non-split groups) in \cite{Fintzen2015}, for example, following \cite{BruhatTits1984}.  In this section, we often use a more concrete description of the Bruhat-Tits building of our groups in terms of lattice chains, as in \cite{BroussousLemaire2002} and \cite{BroussousStevens2009}.

Given a nilpotent element $X$, form a  Lie triple $\LieTrip =\{Y,H,X\}$ and define the subset
$$
\buil(Y,H,X) =\buil(\LieTrip) = \{x \in \buil(G) \mid \LieTrip \subset \LieG_{x,0}\}.
$$
This is a union of facets of $\buil(G)$; let $\facet$ be such a facet.  Then the pair $(\facet,X)$ (or rather, $(\facet,v)$ where $v \in \LieG_{x,0}/\LieG_{x,0+}$ is the image of $X$ in the quotient, though we will not need this here) is called \emph{degenerate}.
We say a degenerate pair $(\facet,X)$ is \emph{distinguished} if $\facet$ is of maximal dimension in $\buil(\LieTrip)$.  DeBacker shows that the rational nilpotent orbits are in bijection with classes of distinguished pairs relative to an equivalence relation called $0$-associativity \cite{DeBacker2002}.  
For ease of notation, write
\[
  \cB(G)^\phiR := \cB(G)^{\phi(\SL_2(\foo))}. 
\]
Then by \cite[Corollary 4.5.5]{DeBacker2002}, we have
$
\cB(G)^\phiR = \cB(\LieTrip).
$

The following theorem characterizes $\cB(G)^\phiR$ in terms of the building of the centralizer $G^\phi$ of $\phi(\SL_2(\ratk))$ for many groups.  Let $D$ denote a central division algebra over $\ratk$.

\begin{theorem}\label{T:building}
  Suppose $G$ is $\GL_n(D)$, $\SL_n(D)$ or a classical group, and suppose
  $\LieTrip=\set{Y,H,X}$ is a Lie triple in $\LieG$.  
  Then there is a natural  $G^\phi$-equivariant identification 
  \[
    \cB(G)^\phiR = \cB(G^\phi). 
  \]
\end{theorem}

\begin{proof}
First let $G=\GL(V,D)$ for some central division algebra  $D$ over $\ratk$ and 
free rank $n$ right $D$-module $V$.  Let $\foo_D\supset \foo$ be the ring of integers
of $D\supset \ratk$.     
We identify $\cB(G)$ with the set of lattice functions
$\LattD(V) = \{\cV_x \mid t \mapsto \mathcal{V}_{x,t} \}$ on $V$ \cite{BroussousLemaire2002}.  Since each
$\foo_D$ lattice in $V$ is naturally also an $\foo$-lattice, $\buil(G)$ is canonically identified as a
subset of $\cB(\GL(V,\ratk))$.  Given $\LieTrip$, recall that we have $G^\phi = \prod_i \GL(\mults^i,D)$ so
$\cB(G^\phi) = \prod_i \cB(\GL(\mults^i,D))$.


Now each $x \in \buil(G)^{\phiR}$ corresponds to a lattice function $\cV_x$. 
Define a map
\begin{equation}\label{eq:R}
  \begin{tikzcd}[column sep=0em, row sep=0em]
    \sR\colon & \buil(G)^{\phiR} \ar[rr]& \text{\hspace{2em}}& \cB(G^\phi)& \\
      &\cV_x \ar[mapsto,rr]&& (\iicMx)_i & \text{where }  \iicMxt := \Hom_{\sltwoR}(\cU_i, \cV_{x,t}). 
  \end{tikzcd}
\end{equation}
Note that $\iicMxt$ has $\foo_D$-module structure inherited from that of
$\cV_{x,t}$.  

On the other hand, we define a map 
\begin{equation}\label{eq:E}
  \begin{tikzcd}[column sep=0em, row sep=0em]
    \sE\colon & \cB(G^\phi) \ar[rr]& \text{\hspace{2em}}& \cB(G) &\\
      &(\iicMx)_i\ar[mapsto,rr]&& \cV_{x} & \text{where }  \cV_{x,t} :=  \bigoplus_i \cU_i \otimes_\foo \iicMxt.
    \end{tikzcd}
\end{equation}
Its image lies in $\buil(G)^\phiR$ by construction.


It is immediate that  for all $y\in \buil(G^\phi)$ we have $\sR(\sE(y)) = y$.
On the other hand, Lemma~\ref{lem:key} ensures that for all $x\in \buil(G)^\phiR$ we have $ \sE(\sR(x)) = x$.  The theorem now follows for $G=\GL_n(D)$, and for $\SL_n(D)$ by restricting the maps $\sR$ and $\sE$.

Now let $G$ be a classical group; that is, for some central division algebra $D$ and sesquilinear form $F\colon V\otimes V\rightarrow D$ we have that  $G=\rU(V,F)$.
As in Section~\ref{SS:centralizers}, given $\LieTrip$ we have
$G^\phi = \prod_i  \rU(\mults^i, \wtf^i)$.
We may identify $\cB(G)$ and $\cB(G^\phi)$) with self-dual lattice functions in
$\cB(\GL(V,D))$ and $\prod_i \cB(\GL(\mults^i,\wtf^i))$ respectively
\cite{BroussousStevens2009}. 
It is immediate that \eqref{eq:E} restricts to a well-defined map $\sE$ from $\buil(G^\phi)$ to $\buil(G)^\phiR$.

To prove that the image of the map $\sR$ of \eqref{eq:R} lies in $\buil(G^\phi)$, suppose that $x\in \cB(G)^\phi \subset \cB(\GL(V,D))^\phi$ and $\sR(x) = (\iicMx)_i \in \buil(\GL(\mults^i,\wtf^i))$.   We need to show that for each $i$ and $x$, 
  $\iicMx$ is a self-dual lattice in $\mults^i$ under the form $\wtf^i$. 
   
  Since $\cV_x$ is self-dual and $\slf_i(\cU_i,\cU_i) = \R$ by assumption, we have that
  \[
    \wtf^i(\iicMxt,\iicMxmtp) = (\slf_i\otimes \wtf^i)(\cU_i \otimes
      \iicMxt,\cU_i\otimes \iicMxmtp) \subset F(\cV_{x,t},\cV_{x,-t^+}) =
    \PP_D.
  \]
  Therefore $\iicMxt \subset (\iicMxmtp)^*$.  On the other
  hand,  the
  pairing $F$ between different the isotypic components
  $\slmod_i\otimes \mults^i$ being zero, we have
  \[
    F(\cU_i\otimes (\iicMxmtp)^*,\cV_{x,-t^+}) = F(\cU_i\otimes
      (\iicMxmtp)^*,\cU_i\otimes \iicMxmtp) = \PP_D.
  \]
 We thus deduce that $\cU_i\otimes (\iicMxmtp)^* \subset (\cV_{x,-t^+})^* =
  \cV_{x,t}$, that is, $(\iicMxmtp)^*\subset \iicMxt$, as required.
The desired identification of $\buil(G)^\phiR$ and $\buil(G^\phi)$ now follows from \eqref{eq:R} and \eqref{eq:E}.
\end{proof}

\section{Realization of the DeBacker parametrization for orthogonal groups} \label{S:minimal}
We now return to the setting of orthogonal groups. 
In this section, we attach facets of the building of $G$ to selected explicit Lie triples of Section~\ref{S:representatives} and use the results of Section~\ref{S:functoriality} to prove these are distinguished representatives which realize the DeBacker correspondence.  We return to the more abstract realization of buildings used in \cite{DeBacker2002}.

Let $T$ be the maximal split torus of $G=\SO(q)$ with Lie algebra $\LieT$.  We have the root system $\Phi=\Phi(G,T)$ with simple system $\Delta = \{\ep_1-\ep_2, \ldots, \ep_{m-1}-\ep_m, \ep_m\}$ if $\ani>0$ and $\Delta = \{\ep_1-\ep_2, \ldots, \ep_{m-1}-\ep_m, \ep_{m-1}+\ep_m\}$ if $\ani=0$.  Let $\apart = \apart(T)$ be the corresponding apartment in $\buil(G)=\buil(\SO(q))$; this is the affine space under $X_*(T)\otimes_{\mathbb{Z}}\mathbb{R}$ on which the roots act by functionals, together with the simplicial structure defined by the affine root hyperplanes $\Hyp_{\alpha,n}=\{x\in \apart \mid \alpha(x)=n\}$, as $\alpha$ ranges over $\Phi$ and $n$ over $\mathbb{Z}$.

We choose a pinning of $G$ relative to $T$, which is a consistent choice of valuation on each root subgroup (or equivalently, root subalgebra), and identified with a choice of (hyperspecial) vertex $x_0$ of $\apart \subset \buil(G)$.  For each $\alpha\in \{ \pm \ep_i \pm \ep_j \mid 1\leq i \neq j \leq m\}$, we have $\dim(\LieG_\alpha)=1$ and we declare that our chosen root vectors $\X_{\pm i, \mp j}$ have valuation $0$.  If $\ani>0$ then there are roots $\alpha\in \{\pm \ep_i \mid 1\leq i \leq m\}$ and for each one, $\dim(\LieG_{\ep_i})=\ani$. 
Following the process described in \cite[\S2]{Fintzen2015}, one determines that a consistent pinning assigns valuation $0$ to each root vector $\X_{\pm i}^\ell$ such that $\val(r_\ell) = 0$ and valuation $\frac12$ to each root vector $\X_{\pm i}^\ell$ such that $\val(r_\ell)=1$.  Then the corresponding $\R$-subalgebra $\LieG_{x_0,0}$ of $\LieG$ (which is the stabilizer in $\LieG$ of the lattice chain attached to $x_0$, in the language of Section~\ref{S:functoriality}) is generated by the $\R$-span of our chosen root vectors.  Its intersection with $\LieT+\mathfrak{s}$ is an $\R$-subalgebra containing, in particular, the $\R$-span of $\{\Hi \mid 1\leq i \leq m\}$.

We put coordinates on $\apart$ so that the vertex $x_0$ is the origin.  
Thus when a vector $X\in \LieG$ is expressed as a linear combination of nonzero vectors in different root spaces  $\sum_{\alpha \in \Phi_X} X_\alpha$, then the $x\in \apart$ for which $X \in \LieG_{x,0}$ are simply described by the condition that for each $\alpha\in \Phi_X$ we have $\val(X_\alpha)+\alpha(x)\geq 0$.

\begin{proposition}\label{P:dimension}
Let $(\partt,\qtup)\in \Nilp(\cls{q},n)$  and let $\Gamma$ be a partition of $\Ip$ as in Proposition~\ref{P:partition}.  Let $\LieTrip_\Gamma$ be an associated Lie triple as constructed in Section~\ref{S:representatives}.  Let $\facet$ denote a maximal facet in
$\buil(\LieTrip_\Gamma)\cap \apart$. 
If $\partt$ is very even, then $\dim(\facet) = \frac12 \vert \partt \vert$ and this is the same value obtained for both $SO(q)$-orbits attached to $\partt$.  Otherwise, $\dim(\facet) = \vert \Gamma_\labeleven\vert + \vert \Gamma_\labelhyp \vert$.
\end{proposition}

\begin{proof}
Suppose that $\LieTrip_\Gamma = \{Y,H,X\}$ is a Lie triple produced in Section~\ref{S:representatives} from a choice of partition $\Gamma$ of $\Ip$.  Then $H \in \spn_{\R}\{\HH_1, \ldots, \HH_n\}$, so it lies in $\LieG_{x,0}$ for all $x \in \apart$.  Let $\Phi_X$ be the set of roots such that for some $\gamma \in \Gamma$, $X|_{V_\gamma}$ has a nonzero projection onto the root space $\LieG_{\alpha}$.  Then we have determined an expression of the form $X = \sum_{\alpha \in \Phi_X} X_{\alpha}$ with each $X_\alpha$ denoting an element of $\LieG_{\alpha}$.  

Reviewing the construction reveals that for our choice of $X$ and $Y$, we have $\Phi_X = -\Phi_Y$, that is, $Y=\sum_{\alpha\in \Phi_X}Y_{\alpha}$ for some nonzero $Y_\alpha \in \LieG_{-\alpha}$.  
We now list, in Table~\ref{Table:rootvectors}, all the pairs $(X_\alpha, Y_\alpha)$, up to multiplication by scalars in $\R^\times$, which appear in the expressions for $X$ and $Y$ in \eqref{E:hypX}, \eqref{E:pairsX}, \eqref{E:quadX}, \eqref{E:tripX}, \eqref{E:signX}, \eqref{E:aniX} and \eqref{E:veryevenX}.  In doing so, we make use of the fact that the coefficients $\{c,s,2,-1\}$ lie in $\R^\times$ but that the coefficients $r_{i,j}$ and $r_\ell$ (often abbreviated as $r$) variously take values in $\{\R^\times, \varpi \R^\times\}$.   Given that $\val(\Ximj)=\val(\Xipj)=0$ and $\val(\Xil)=\frac12\val(r_\ell)$, we compute the valuations of $X_\alpha$ and $Y_\alpha$ in the last two columns of Table~\ref{Table:rootvectors}.  

\begin{table}[ht]
\begin{tabular}{ccccc}
$\alpha \in \Phi^+$ & $X_\alpha$ & $Y_{\alpha}$ & $\val(X_\alpha)$ & $\val(Y_\alpha)$ \\
\hline
$\ep_i-\ep_{i+1}$ & $\X_{i,i+1}$ & $\X_{i+1,i}$ & $0$ & $0$ \\
$\ep_i + \ep_j$ & $\X_{i,-j}$ & $\X_{-i,j}$ & $0$&$0$\\
 $\ep_i + \ep_j$ & $\varpi \X_{i,-j}$ & $\varpi^{-1}\X_{-i,j}$ & $1$ & $-1$\\
 $\ep_i$ & $\Xil$ & $r_\ell^{-1}\X_{-i}^\ell$ & $\frac12\val(r_\ell)$ & $-\frac12\val(r_\ell)$
 \end{tabular}
\caption{Nonzero pairs $(X_\alpha,Y_\alpha)$, with $\alpha \in \Phi_X$, such that the projections of $X$ and $Y$ onto the root spaces $\LieG_{\alpha}$ and $\LieG_{-\alpha}$ lie in $\R^\times X_\alpha$ and $\R^\times Y_\alpha$, respectively.  The final columns record the valuations of $X_\alpha$ and $Y_\alpha$, respectively.} \label{Table:rootvectors}
\end{table}

Note that $X,Y\in \LieG_{x,0}$ if and  only if $\val(X_\alpha)\geq -\alpha(x)$ and $\val(Y_\alpha)\geq \alpha(x)$ for all $\alpha$.  We observe from Table~\ref{Table:rootvectors} that our representatives satisfy $\val(Y_{\alpha})=-\val(X_\alpha)$ for all $\alpha$.  Hence these were optimally chosen: $\LieTrip_\Gamma\in \LieG_{x,0}$ if and only if 
\begin{equation}\label{E:system}
\alpha(x)=-\val(X_\alpha), \quad \forall \alpha \in \Phi_X.
\end{equation}
We claim this system of equations is consistent, whereby the solution is an affine subspace $\apart_\Gamma$ which is a union of facets of $\apart$.  In this case we'll let $\facet$ be any maximal facet of $\apart_\Gamma$; note that $\dim(\facet)=\dim(\apart_\Gamma)$.
 


To solve \eqref{E:system} explicitly, let $\tilde{\Gamma}$ be the partition of the set $\{1, 2, \cdots, m\}$ induced by the partition $\Gamma$; that is, for each $\gamma \in \Gamma$ the element $\tilde{\gamma}\in \tilde{\Gamma}$ is the set of indices of the Witt basis $B_\gamma$ attached to $V_\gamma$.  For each $\gamma \in \Gamma$, the roots $\alpha \in \Phi_{X|_{V_\gamma}}$ are linear combinations of $\{\ep_i \mid i \in \tilde{\gamma}\}$.  The linear system \eqref{E:system} can thus be decoupled into $\vert \Gamma \vert$ distinct linear systems, whence any solution may be written $x = \sum_{\tilde{\gamma}\in \tilde{\Gamma}}x_{\tilde{\gamma}}$ as each $x_{\tilde{\gamma}}$ runs over an affine set $\apart_{\tilde{\gamma}}$.  Let $e_{\tilde{\gamma}}$ be the vector such that $\ep_i(e_{\tilde{\gamma}})=1$ for all $i\in \tilde{\gamma}$ and $0$ otherwise.

It is then a straightforward exercise to see that solving these uncoupled systems (with notation of the corresponding paragraphs in Section~\ref{S:representatives}) yields
$$
\apart_{\tilde{\gamma}} = \begin{cases}
\mathbb{R}e_{\tilde{\gamma}} & \text{if $\gamma \in \Gamma_\labeleven \cup \Gamma_\labelhyp$},\\
\{-\frac12\val(r)e_{\tilde{\gamma}}\} & \text{if $\gamma \in \Gamma_\labelpairs\cup \Gamma_\labelfours \cup\Gamma_\labeltriple$},\\
 \{-\frac12\val(r_\ell)e_{\tilde{\gamma}}\} & \text{if $\gamma \in \Gamma_\labelsign\cup \Gamma_\labelani$.}
 \end{cases}
$$
We conclude that $\dim(\apart_\Gamma)=\dim(\facet) = \vert \Gamma_\labeleven \vert + \vert \Gamma_\labelhyp \vert$, proving the final statement of the proposition.

Now suppose that $\partt$ is very even, so that there are two $\SO(q)$-orbits.  The first is constructed as in Section~\ref{S:hyp}.  Since $\Gamma = \Gamma_\labeleven$ consists entirely of pairs, it has size $\dim(\facet)=\frac12 \vert \Ip\vert = \frac12 \vert \partt \vert$, which gives the stated dimension for the first $\SO(q)$-orbit.  The second $\SO(q)$-orbit is obtained by modifying the choice of $X|_{V_\gamma}$ on one $\gamma \in \Gamma_\labeleven$ with an expression of the form \eqref{E:veryevenX}.  The corresponding linear system 
consistent, with solution space $\mathbb{R}\tilde{e}$ where $\ep_j(\tilde{e})=0$ if $j \notin \tilde{\gamma}$, $\ep_j(\tilde{e})=1$ if $1\leq j<i$, and $\ep_{i}(\tilde{e})=-1$.  As the solution to this subsystem is again one-dimensional, the dimension of the facet is the same for both orbits attached to $\partt$.
\end{proof}

Recall that by definition, a pair $\gamma = \{(j_t,m), (j_t,m')\}$ in $\Gamma_\labelhyp$  corresponds to a hyperbolic plane in $q_{j_t}$.  Therefore to maximize $\vert \Gamma_\labeleven\vert + \vert \Gamma_\labelhyp \vert$ over all partitions $\Gamma$ of $\Ip$ is to choose a partition which contain all such hyperbolic planes, and all the even parts, from $\Ip$.  Then $\Ip \setminus (\Gamma_\labeleven \cup \Gamma_\labelhyp)$ has precisely $\sum_{i=1}^s \da{q_{j_i}}$ elements, corresponding to the anisotropic kernels of all of the quadratic forms in $\qtup$.  Complete this to a partition $\Gammamax$ of $\Ip$ satisfying Proposition~\ref{P:partition}.

\begin{theorem}\label{T:DeBackercorrespondence}
Let $(\partt,\qtup) \in \Nilp(\cls{q},n)$ or $\Nilps(n)$, and $\Gammamax$ a partition of $\Ip$ as above.  Then the corresponding pair $(\facet_\Gammamax, X_\Gammamax)$ is distinguished. Its associativity class is the unique one attached to the rational nilpotent orbit $G\cdot X_\Gammamax$ by the DeBacker correspondence, for $G = \Orth(q)$ or $G=\SO(q)$.
\end{theorem}

\begin{proof}
Since for each $\Gamma$, each pair in $\Gamma_\labeleven\cup\Gamma_\labelhyp$ contributes one linear degree of freedom to the solution space $\apart_{\Gamma}$, we conclude that for our choice of $\Gamma =\Gammamax$ that
$$
\dim(\facet_\Gammamax) =\frac12\left(\vert \partt \vert-\sum_{i=1}^s \da{q_{j_i}} \right).
$$
On the other hand, by Theorem~\ref{T:building} we have $\dim(\buil(\LieTrip))=\dim(\buil(G^\phi))$, which is equal to the split rank of $G^\phi$.  The split rank of $\rU(\mults^i,\wtf^i)$ is exactly the maximal number of orthogonal hyperbolic planes in $\wtf^i$; for $i$ odd this is therefore $\frac12(\deg(\wtf^i)-\da{\wtf^i})$. We deduce from \eqref{E:cent} that
$$
\dim(\facet_\Gammamax) = \dim (\buil(\LieTrip)),
$$
so the facet is indeed distinguished, as required.
\end{proof}

This theorem establishes a constructive map from the classical partition-based parametrization of nilpotent orbits of orthogonal and special orthogonal groups to the building-based parametrization proposed by DeBacker.
 
\bibliographystyle{amsplain}

\begin{thebibliography}{10}

\bibitem{Ahlenetal2018}
Olof Ahl\'en, Henrik~P.A. Gustafsson, Axel Kleinschmidt, Liu Baiying, and
  Daniel Persson, \emph{Fourier coefficients attached to small automorphic
  representations of $\mathrm{SL}_n(\mathbb{A})$},  J. Number Theory 192 (2018), 80–142. 
 

\bibitem{TBernstein2015}
Tobias Bernstein, \emph{A classification of $p$-adic quadratic forms}, Preprint
  available at \url{https://bit.ly/2ABZwf9}, 2015.

\bibitem{Bourbaki1975}
N.~Bourbaki, \emph{\'{E}l\'ements de math\'ematique. {F}asc. {XXXVIII}:
  {G}roupes et alg\`ebres de {L}ie. {C}hapitre {VII}: {S}ous-alg\`ebres de
  {C}artan, \'el\'ements r\'eguliers. {C}hapitre {VIII}: {A}lg\`ebres de {L}ie
  semi-simples d\'eploy\'ees}, Actualit\'es Scientifiques et Industrielles, No.
  1364. Hermann, Paris, 1975.

\bibitem{BroussousLemaire2002}
Paul Broussous and Bertrand Lemaire, \emph{Building of {${\rm GL}(m,D)$} and
  centralizers}, Transform. Groups \textbf{7} (2002), no.~1, 15--50.

\bibitem{BroussousStevens2009}
Paul Broussous and Shaun Stevens, \emph{Buildings of classical groups and
  centralizers of {L}ie algebra elements}, J. Lie Theory \textbf{19} (2009),
  no.~1, 55--78.

\bibitem{BruhatTits1984}
Fran\c{c}ois Bruhat and Jacques Tits, \emph{Groupes r\'eductifs sur un corps
  local. {II}. {S}ch\'emas en groupes. {E}xistence d'une donn\'ee radicielle
  valu\'ee}, Inst. Hautes \'Etudes Sci. Publ. Math. (1984), no.~60, 197--376.

\bibitem{Carter1985}
Roger~W. Carter, \emph{Finite groups of {L}ie type}, Wiley Classics Library,
  John Wiley \& Sons Ltd., Chichester, 1993, Conjugacy classes and complex
  characters, Reprint of the 1985 original, A Wiley-Interscience Publication.

\bibitem{Christie}
Aaron Christie, \emph{Fourier {E}igenspaces of {W}aldspurger's {B}asis},
  Preprint \texttt{arxiv.org} arXiv:1411.1037v2, 2014.

\bibitem{CMcG1993}
David~H. Collingwood and William~M. McGovern, \emph{Nilpotent orbits in
  semisimple {L}ie algebras}, Van Nostrand Reinhold Mathematics Series, Van
  Nostrand Reinhold Co., New York, 1993.

\bibitem{DeBacker2002}
Stephen DeBacker, \emph{Parametrizing nilpotent orbits via {B}ruhat-{T}its
  theory}, Ann. of Math. (2) \textbf{156} (2002), no.~1, 295--332.

\bibitem{Diwadkar2008}
Jyotsna~Mainkar Diwadkar, \emph{Nilpotent conjugacy classes in {$p$}-adic {L}ie
  algebras: the odd orthogonal case}, Canad. J. Math. \textbf{60} (2008),
  no.~1, 88--108.

\bibitem{Fintzen2015}
Jessica Fintzen, \emph{On the {M}oy-{P}rasad filtration}, Preprint
  \texttt{arxiv.org} arXiv:1511.00726v3 [math.RT], 2017.

\bibitem{FrechetteGordonRobson2015}
Sharon~M. Frechette, Julia Gordon, and Lance Robson, \emph{Shalika {G}erms for
  $\mathfrak{sl}_n$ and $\mathfrak{sp}_{2n}$ {A}re {M}otivic}, Women in Numbers
  Europe, Association for Women in Mathematics Series, vol 2., Springer, Cham,
  2015, Bertin M., Bucur A., Feigon B., Schneps L. (eds), pp.~51--85.

\bibitem{JiangLiu2015}
Dihua Jiang and Baiying Liu, \emph{On special unipotent orbits and {F}ourier
  coefficients for automorphic forms on symplectic groups}, J. Number Theory
  \textbf{146} (2015), 343--389.

\bibitem{Lam2005}
Tsit-Yuen~(T.Y.) Lam, \emph{Introduction to quadratic forms over fields},
  Graduate Studies in Mathematics, vol.~67, American Mathematical Society,
  Providence, RI, 2005.

\bibitem{McNinch2005}
George~J. Mc{N}inch, \emph{Optimal {$\mathrm{SL}(2)$}-homomorphisms}, Comment.
  Math. Helv. \textbf{80} (2005), 391--426.

\bibitem{McNinch2018}
\bysame, \emph{On the nilpotent orbits of a reductive group over a local
  field}, Preprint, Author's webpage,
  https://gmcninch.math.tufts.edu/assets/manuscripts/2019-On-the-nilpotent-orbits-of-a-reductive-group-over-a-local-field.pdf,
  2018.

\bibitem{Nevins2002}
Monica Nevins, \emph{Admissible nilpotent orbits of real and {$p$}-adic split
  exceptional groups}, Represent. Theory \textbf{6} (2002), 160--189.

\bibitem{Nevins2011a}
\bysame, \emph{On nilpotent orbits of {${\rm SL}_n$} and {${\rm Sp}_{2n}$} over
  a local non-{A}rchimedean field}, Algebr. Represent. Theory \textbf{14}
  (2011), no.~1, 161--190.

\bibitem{Tits1979}
Jacques Tits, \emph{Reductive groups over local fields}, Automorphic forms,
  representations and {$L$}-functions ({O}regon {S}tate {U}niv., {C}orvallis,
  {O}re., 1977), {P}art 1, Proc. Sympos. Pure Math., XXXIII, Amer. Math. Soc.,
  Providence, R.I., 1979, pp.~29--69.

\bibitem{Waldspurger2001}
Jean-Loup Waldspurger, \emph{Int\'egrales orbitales nilpotentes et endoscopie
  pour les groupes classiques non ramifi\'es}, Ast\'erisque (2001), no.~269,
  vi+449.

\bibitem{Yap2018}
Jit~Wu Yap, \emph{On {D}e{B}acker's parametrization of rational nilpotent
  orbits of $\mathbb{O}_{2n}$}, Preprint, Overseas UROPS report under the
  supervision of Jia-Jun Ma,
  \url{https://www.majiajun.org/pdfs/UROPSR-JitWu.pdf}, 2018.

\end{thebibliography}
\providecommand{\bysame}{\leavevmode\hbox to3em{\hrulefill}\thinspace}
\providecommand{\MR}{\relax\ifhmode\unskip\space\fi MR }
\providecommand{\MRhref}[2]{%
  \href{http://www.ams.org/mathscinet-getitem?mr=#1}{#2}
}
\providecommand{\href}[2]{#2}

\end{document}